\documentclass[9pt]{article}
\usepackage{listings}
\usepackage{pdflscape}
\usepackage{amsfonts}
\usepackage{epsfig}
\usepackage{lscape}
\usepackage{longtable}
\usepackage{amsmath,amssymb,amsthm}
\usepackage{graphicx}
\usepackage{subfigure}
\usepackage{color}
\usepackage{placeins}
\usepackage{url}
\usepackage{cases}
\usepackage{longtable}
\usepackage{supertabular}
\usepackage{multirow}
\newcommand{\tabincell}[2]{\begin{tabular}{@{}#1@{}}#2\end{tabular}}
\usepackage{amsmath}
\allowdisplaybreaks[4]
\usepackage{cite}
\usepackage{algorithmic}
\usepackage{algorithm,float}
\usepackage{booktabs}
\usepackage[shortlabels]{enumitem}  
\usepackage[online]{threeparttablex}
\usepackage{siunitx}
\usepackage{tikz}



\newtheorem{theorem}{Theorem}

\newtheorem{remark}{Remark}

\newtheorem{proposition}{Proposition}

\newenvironment{varalgorithm}[1]
{\algorithm[H]\renewcommand{\thealgorithm}{#1}}
{\endalgorithm}

\begin{document}
\title{\bf \Large{Efficient algorithms for multivariate shape-constrained convex regression problems }\footnotemark[1]}
\author{Meixia Lin\footnotemark[2], \quad Defeng Sun\footnotemark[3], \quad Kim-Chuan Toh\footnotemark[4]}

\date{January 07, 2020}
\maketitle

\renewcommand{\thefootnote}{\fnsymbol{footnote}}
\footnotetext[1]{{\bf Funding}: Defeng Sun is supported in part by Hong Kong Research Grant Council grantPolyU153014/18p and Kim-Chuan Toh by the Academic Research Fund (grant R-146-000-257-112) of the Ministry of Education, Singapore.}
\footnotetext[2]{Department of Mathematics, National University of Singapore, 10 Lower Kent Ridge Road, Singapore ({\tt lin\_meixia@u.nus.edu}).}
\footnotetext[3]{Department of Applied Mathematics, The Hong Kong Polytechnic University, Hung Hom, Hong Kong ({\tt defeng.sun@polyu.edu.hk}).}
\footnotetext[4]{Department of Mathematics and Institute of Operations Research and Analytics, National University of Singapore, 10 Lower Kent Ridge Road, Singapore ({\tt mattohkc@nus.edu.sg}).}
\renewcommand{\thefootnote}{\arabic{footnote}}

\begin{abstract}
	Shape-constrained convex regression problem deals with fitting a convex function to the observed data, where additional constraints are imposed, such as component-wise monotonicity and uniform Lipschitz continuity. This paper provides a comprehensive mechanism for computing the least squares estimator of a multivariate shape-constrained convex regression function in $\mathbb{R}^d$. We prove that the least squares estimator is computable via solving a constrained convex quadratic programming (QP) problem with $(n+1)d$ variables and at least $n(n-1)$ linear inequality constraints, where $n$ is the number of data points. For solving the generally very large-scale convex QP, we design two efficient algorithms, one is the symmetric Gauss-Seidel based alternating direction method of multipliers ({\tt sGS-ADMM}), and the other is the proximal augmented Lagrangian method ({\tt pALM}) with the subproblems solved by the semismooth Newton method ({\tt SSN}). Comprehensive numerical experiments, including those in the pricing of basket options and estimation of production functions in economics, demonstrate that both of our proposed algorithms outperform the state-of-the-art algorithm. The {\tt pALM} is more efficient than the {\tt sGS-ADMM} but the latter has the advantage of being simpler to implement.
\end{abstract}

\medskip
\noindent
{\bf Keywords:} convex regression, shape constraints, preconditioned proximal point algorithm, symmetric Gauss-Seidel based ADMM
\\[5pt]
{\bf AMS subject classification:} 90C06, 90C25, 90C90

\section{Introduction}\label{sec:introduction}
Convex (or concave) regression is meant to estimate a convex (or concave) function based on a finite number of observations. It is a topic of interest in many fields such as economics, operations research and financial engineering. In economics, production functions \cite{hildreth1954point,varian1984nonparametric,allon2007nonparametric}, demand functions \cite{varian1982nonparametric} and utility functions \cite{meyer1968consistent} are often required to be concave. In operations research, the performance measure expectations can be proved to be convex in the underlying model parameters, e.g. in the context of queueing network \cite{chen2013fundamentals}. In financial engineering, the option pricing function has the convexity restriction under the no-arbitrage condition, as can be seen from \cite{ait2003nonparametric}.

In the literature, there exists various methods for solving the convex regression problem. With the specification of a functional form, one can apply a parametric approach to estimate the convex function. For example, the Cobb-Douglas production function is a particular functional form of the production function that is widely used in applied production economics. To avoid strong prior assumptions on the functional form, one can also use a non-parametric approach to perform the function estimation. Generally, the nonparametric estimation is based on a given collection of primitive functions, such as local polynomial \cite{longstaff2001valuing}, trigonometric series, spline estimator \cite{dontchev2003quadratic,qi2007regularity} and kernel-type estimator \cite{aybat2014parallel}. However, such an approach may face some difficulties such as imposing the convexity constraint and choosing appropriate smoothing parameters (e.g. the degree of the polynomial, or the kernel density bandwidth). To overcome these difficulties, we use the least squares estimator for the convex regression. The least squares estimator is first proposed in \cite{hildreth1954point}, and its theoretical properties are carefully studied in \cite{hanson1976consistency,seijo2011nonparametric,lim2012consistency}. 

Suppose that we observe $n$ data points $\{(X_i,Y_i)\}_{i=1}^n$, which satisfy the regression model $Y = \psi(X)+\varepsilon$ with an unknown convex function $\psi:\mathbb{R}^d \rightarrow \mathbb{R}$. The least squares estimation method is to estimate the function $\psi$ by minimizing the sum of squares error $\sum_{i=1}^n(\psi(X_i)-Y_i)^2$ over the set of convex functions from $\mathbb{R}^d$ to $\mathbb{R}$. This infinite dimensional model appears to be intractable. Fortunately, the authors in \cite{kuosmanen2008representation,seijo2011nonparametric} have provided a computationally tractable optimal solution to it. They showed that in the convex regression problem, the family of convex functions can be characterized by a subset of continuous, piecewise linear functions $\theta_i + \langle \xi_i,X-X_i\rangle$, $i=1,\ldots,n$, whose intercepts $\theta_i$'s and gradient vectors $\xi_i$'s are restricted to satisfy the convexity conditions. The resulting problem is a convex quadratic programming (QP) problem with $(n+1)d$ variables and $n(n-1)$ linear inequality constraints, which can be solved by interior point solvers such as those implemented in {\tt MOSEK} when $n$ is not too large, as stated in \cite{seijo2011nonparametric}. However, interior point solvers may 
easily run out of memory when $n$ is large due to the presence of a large number of at least $n(n-1)$ linear inequality constraints. Aybat et al. \cite{aybat2014parallel} proposed a parallel proximal gradient method ({\tt PAPG}) to solve the dual of an approximation of the QP by adding a ridge regularization on the $\xi_i$'s. The {\tt PAPG} method however is not fast enough for solving large problems. For example, it needs $17$ minutes to solve a problem with $d=80$, $n=1600$ on a $16$-core machine sharing $32$ GB. Mazumder et al. \cite{mazumder2019computational} proposed a three-block alternating direction method of multipliers ({\tt ADMM}) for solving the QP, but it has no convergence guarantee. The computational challenge of handling large-scale cases still remains in need of more progress, especially for the case when $d$ and $n$ are relatively large, for which existing methods are too expensive even for computing a solution with low accuracy.

In many real applications, one may need to impose more shape constraints on the convex function $\psi$, such as component-wise monotonicity and uniform Lipschitz continuity. For example, the option pricing function under the no-arbitrage condition needs to be non-decreasing as well as convex as described in \cite{ait2003nonparametric}. In addition, when dealing with the Lipschitz convex regression as in \cite{lim2014convergence,balazs2015near,mazumder2019computational}, the uniform Lipschitz property of the convex function is added when performing the estimation. To deal with these cases, we minimize the sum of squares error $\sum_{i=1}^n(\psi(X_i)-Y_i)^2$ over the set of convex functions satisfying additional shape constraints. The addition of the shape constraints obviously would make the QP even more complicated and difficult to solve.

In this paper, we provide a comprehensive mechanism for computing the least squares estimator for the shape-constrained convex regression problem. We first prove 
that the minimal sum of squares error can be achieved via a set of piecewise linear functions whose intercept and gradient vectors are constrained to satisfy the convexity conditions and required shape constraints (see Theorem \ref{thm:relation_two_fun}). This conclusion leads us to a constrained QP with $(n+1)d$ variables, $n(n-1)$ linear inequality constraints and $n$ probably non-polyhedral inequality constraints. Note that the estimator obtained in this way is nonsmooth, one can apply the Moreau proximal smoothing technique to obtain a smoothing approximation. In addition, we can use a generalized form of the proposed constrained QP model as well as a data-driven Lipschitz estimation method to handle the boundary effect of the least squares estimator for convex functions. The main task in this mechanism is to solve the constrained QP in a robust and efficient manner. Most existing methods for the QP in the standard convex regression are either not extendable or difficult to be modified to solve the constrained QP due to the additional shape constraints. For the multivariate shape-constrained convex regression problem, even with only a moderate number of observations, say $n=1000$, the memory cost and computational cost are already massive since the underlying QP has about a million constraints. To tackle the potentially very large-scale QPs, we design two algorithms that can fully exploit the underlying structures of the convex QPs of interest. The first algorithm is the symmetric Gauss-Seidel based alternating direction method of multipliers ({\tt sGS-ADMM}), which has been demonstrated to perform better than the possibly nonconvergent directly extended multi-block {\tt ADMM} \cite{chen2017efficient}. The {\tt sGS-ADMM} algorithm is easy to implement, but it is still just a first-order method, which may not be efficient enough to solve a problem to high accuracy. We also design a proximal augmented Lagrangian method ({\tt pALM}) for solving the constrained QP, which is proved to be superlinearly convergent. For the {\tt pALM} subproblems, we solve them by the semismooth Newton method ({\tt SSN}), which is proved to have quadratic convergence. Moreover, we fully uncover and exploit the second order sparsity structure of the problem to highly reduce the computational cost of solving the Newton systems. Comprehensive numerical experiments, including those in the pricing of basket options and estimation of production functions, demonstrate that both of our proposed algorithms outperform the state-of-the-art algorithm.

\section{Model with shape constraints}\label{sec:model}

Given independent observations $\{(X_i,Y_i)\}_{i=1}^n$, where the predictors $X_i\in\mathbb{R}^d$ and the responses $Y_i\in \mathbb{R}$, we aim to fit a convex regression model of the form $Y = \psi(X)+\varepsilon$. In the model, $\psi:\Omega \rightarrow \mathbb{R}$ is an unknown convex function, $\Omega\subset \mathbb{R}^d$ is a $\delta$-neighborhood of ${\rm conv}(X_1,\cdots,X_n)$ (the convex hull of $\{X_i\}_{i=1}^n$), $\varepsilon$ is a random variable with expectation $\mathbb{E}[\varepsilon\vert X]=0$. The least squares estimator $\hat{\psi}$ of $\psi$ is defined as
\begin{align}\label{min_function}
\hat{\psi}\in \underset{\psi\in\mathcal{C}}{\arg\min}\sum_{i=1}^n(\psi(X_i)-Y_i)^2,\quad \mathcal{C}=\{ \psi:\Omega \rightarrow \mathbb{R}\mid \psi\mbox{ is a convex function} \}.
\end{align}
The authors in \cite{kuosmanen2008representation,seijo2011nonparametric} provided a computationally tractable optimal solution to the above infinite dimensional model. Specifically, once an optimal solution $\{(\hat{\theta}_i,\hat{\xi}_i)\}_{i=1}^n$ to the following finite dimensional problem
\begin{align*}
\min_{\theta_1,\ldots,\theta_n\in\mathbb{R};  \xi_1,\ldots,\xi_n \in\mathbb{R}^{d}} 
\Big\{ \frac{1}{2} \sum_{i=1}^n (\theta_i - Y_i)^2 \Bigm\lvert  \theta_i \geq \theta_j + \langle \xi_j,X_i - X_j\rangle,\  1\leq i ,j \leq n \Big\}
\end{align*}
has been computed, one can construct an optimal solution $\hat{\psi}$ to \eqref{min_function} by taking
\begin{align} 
\hat{\psi}(x) = \max_{1\leq j\leq n} \Big\{ \hat{\theta}_j + \langle \hat{\xi}_j,x- X_j\rangle 
\Big\}, \quad  x\in \Omega.\label{extension}
\end{align}

For the shape-constrained convex regression problem, the least squares estimator $\hat{\psi}$ is defined as
\begin{align}\label{eq:min_CS}
\hat{\psi}\in \underset{\psi\in\mathcal{C}_\mathcal{S} }{\arg\min}\sum_{i=1}^n(\psi(X_i)-Y_i)^2,\quad \mathcal{C}_\mathcal{S}=\{\psi:\Omega \rightarrow \mathbb{R}\mid \psi\mbox{ is a convex function with Property }\mathcal{S} \},
\end{align}
where Property $\mathcal{S}$ specifies the shape constraint of $\psi$. We restrict ourselves to the case that Property $\mathcal{S}$ takes one of the following forms:
\begin{enumerate}[label=(S\arabic*)]
	\item (monotone constraint) $\psi$ is non-decreasing in some of the coordinates (denoted as $K_1$) and non-increasing in some others (denoted as $K_2$), where $K_1$ and $K_2$ are disjoint subsets of $\{1,\cdots,d\}$;
	\item (box constraint) the elements in $\partial \psi(x)$ for any $x\in \Omega$ are bounded by two given vectors $L,U\in\mathbb{R}^d$;
	\item (Lipschitz constraint) $\psi$ is Lipschitz, i.e., $|\psi(x)-\psi(y)|\leq L \|x-y\|_p$ for any $x,y\in \Omega$, where $p=1,2,\infty$, and $L$ is a given positive constant.
\end{enumerate}

\paragraph{Structure of the paper.} In the remaining part of this paper, we provide the mechanism for estimating the multivariate shape-constrained convex function in Section \ref{sec:mechanism}. For solving the involved constrained QP, the {\tt sGS-ADMM} algorithm is presented in Section \ref{sec:sGSADMM} and the {\tt pALM} algorithm is described in Section \ref{sec:pALM}. The implementation details of the proposed algorithms can be found in Section \ref{sec:proximalmapping}. Section \ref{sec:numerical} provides the numerical comparison among {\tt MOSEK}, {\tt sGS-ADMM} and {\tt pALM}, which demonstrates the robustness and efficiency of {\tt pALM}. Then we apply our mechanism to perform the function estimation in several interesting real applications in Section \ref{sec:realapplication}. Finally, we conclude the paper and discuss some future work.

\paragraph{Notation.} Denote $X=(X_1,\cdots,X_n)\in \mathbb{R}^{d\times n}$ and $e_n=(1,\cdots,1)^T\in \mathbb{R}^n$. For any matrix $Z\in \mathbb{R}^{m\times n}$, $Z_i$ denotes the $i$-th column of $Z$. We use ``${\rm Diag}(z)$" to denote the diagonal matrix whose diagonal is given by the vector $z$. For any positive semidefinite matrix $H\in \mathbb{R}^{n\times n}$, we define $\langle x,x'\rangle_{H}:=\langle x,H x'\rangle$, and $\|x\|_{H}:=\sqrt{\langle x,x\rangle_{H}}$ for all $x,x'\in \mathbb{R}^n$. For a given closed subset $C$ of $\mathbb{R}^n$ and $x\in\mathbb{R}^n$, we define ${\rm dist}_{H}(x,C)=\min\{\|x-y\|_H\mid y\in C\}$. The largest (smallest) eigenvalue of $H$ is denoted as $\lambda_{\max}(H)$ ($\lambda_{\min}(H)$). Given $x\in \mathbb{R}^n$ and an index set $K\subset \{1,\cdots,n\}$, $x_K$ denotes the sub-vector of $x$ with those elements not in $K$ being removed. For a closed proper convex function $q:\mathbb{R}^n\rightarrow (-\infty,\infty]$, the conjugate of $q$ is $q^*(z):=\sup_{x\in\mathbb{R}^n}\{\langle x,z\rangle-q(x)\}$. The Moreau envelope of $q$ at $x$ is defined by
\[
{\rm E}_q(x):=\min_{y\in \mathbb{R}^n}\Big\{ q(y)+\frac{1}{2}\|y-x\|^2\Big\},
\]
and the associated proximal mapping ${\rm Prox}_q(x)$ is
defined as the unique solution of the above minimization problem. As proved in \cite{moreau1965proximite}, ${\rm E}_q(\cdot)$ is finite-valued, convex and differentiable with $\nabla {\rm E}_q(x)=x-{\rm Prox}_q(x)$. In addition, we can see from \cite{nocedal2006numerical,rockafellar1976monotone} that ${\rm Prox}_q(x)$ is Lipschitz continuous with modulus $1$.

\section{A mechanism for estimating the multivariate shape-constrained convex function}\label{sec:mechanism}
In this section, we provide a comprehensive mechanism for computing the least squares estimator for the multivariate shape-constrained convex function defined in \eqref{eq:min_CS}. Before describing the process, we first characterize Property $\mathcal{S}$ in the following proposition. For brevity, we omit the proof.
\begin{proposition}\label{prop:defD}
	A convex function $\psi$ has Property $\mathcal{S}$ if and only if for any $x\in \mathbb{R}^d$, the subdifferential of $\psi$ satisfies $\partial \psi(x)\subset \mathcal{D}$, where $\mathcal{D}$ is defined corresponding to Property $\mathcal{S}$ as follows:
	\begin{enumerate}[label=(S\arabic*)]
		\item (monotone constraint) $\mathcal{D}=\{x\in \mathbb{R}^d\mid x_{K_1}\geq 0,x_{K_2}\leq 0\}$,
		\item (box constraint) $\mathcal{D}=\{x\in \mathbb{R}^d\mid L\leq x\leq U\}$,
		\item (Lipschitz constraint) $\mathcal{D}=\{x\in \mathbb{R}^d\mid \|x\|_q\leq L\}$, where $q$ satisfies $1/p+1/q=1$.
	\end{enumerate}
\end{proposition}

The least squares estimation problem \eqref{eq:min_CS} attempts to find a best-fitting function $\hat{\psi}$ from the function family $\mathcal{C}_\mathcal{S}$, which is infinite dimensional. Therefore, this problem is intractable in practice. In order to design a tractable approach, we establish the following representation theorem to \eqref{eq:min_CS}, which is motivated by \cite{kuosmanen2008representation}. 

\begin{theorem}\label{thm:relation_two_fun}
	Define the set of piecewise linear functions as
	\begin{align}\label{eq:KS}
	\mathcal{K}_\mathcal{S}:=\Big\{\phi:\Omega \rightarrow \mathbb{R}\Bigm\lvert \phi(x)=\max_{1\leq j\leq n}\{  \theta_j+\langle \xi_j,x-X_j\rangle	\},\
	(\theta_1,\cdots,\theta_n,\xi_1,\cdots,\xi_n)\in \mathcal{F}_\mathcal{S}
	\Big\},
	\end{align}
	where 
	\begin{align}\label{eq:FS}
	\mathcal{F}_\mathcal{S}:=\{ (\theta_1,\cdots,\theta_n,\xi_1,\cdots,\xi_n)\mid
	\theta_i\in \mathbb{R},\xi_i\in \mathcal{D},i=1,\cdots,n,\ \theta_i \geq \theta_j + \langle \xi_j,X_i - X_j\rangle,1\leq i ,j \leq n
	\},
	\end{align}
	and $\mathcal{D}$ is defined as in Proposition \ref{prop:defD}. Consider the problem
	\begin{align}\label{min_KS}
	\min_{\phi\in \mathcal{K}_\mathcal{S}}\sum_{i=1}^n(\phi(X_i)-Y_i)^2.
	\end{align}
	Then the following equality holds:
	\begin{align}\label{eq:twomin}
	\min_{\psi\in\mathcal{C}_\mathcal{S}}
	\sum_{i=1}^n(\psi(X_i)-Y_i)^2=\min_{\phi\in\mathcal{K}_\mathcal{S}}
	\sum_{i=1}^n(\phi(X_i)-Y_i)^2.
	\end{align}
	Moreover, any solution $\hat{\phi}$ to \eqref{min_KS}
	is a solution to the problem \eqref{eq:min_CS}.
\end{theorem}
\begin{proof}
We first prove that $\mathcal{K}_\mathcal{S}\subset \mathcal{C}_\mathcal{S}$, that is, the functions in $\mathcal{K}_\mathcal{S}$ are convex functions with Property $\mathcal{S}$. Convexity comes from the fact that any pointwise maximum function is convex. Given $\phi\in \mathcal{K}_\mathcal{S}$ determined by $(\theta_1,\cdots,\theta_n,\xi_1,\cdots,\xi_n)\in \mathcal{F}_\mathcal{S}$, the subdifferential of the piecewise linear function $\phi$ is a polyhedron according to \cite[Theorem 25.6]{rockafellar1970convex}, and it is given by
\begin{align*}
\partial \phi(x)={\rm conv}\{\xi_i\mid i\in I(x)\},\quad I(x):=\{i\mid \theta_i+\langle \xi_i,x-X_i\rangle=\phi(x) \}.
\end{align*}
By the definition of $\mathcal{D}$ and $\mathcal{F}_\mathcal{S}$, we can see that $\partial \phi(x)\subset\mathcal{D}$ for any $x\in \Omega$. According to Proposition \ref{prop:defD}, the convex function $\phi$ has Property $\mathcal{S}$, which means $\phi\in\mathcal{C}_\mathcal{S}$. Therefore, we have that
\begin{align*}
\min_{\psi\in\mathcal{C}_\mathcal{S}}
\sum_{i=1}^n(\psi(X_i)-Y_i)^2\leq\min_{\phi\in\mathcal{K}_\mathcal{S}}
\sum_{i=1}^n(\phi(X_i)-Y_i)^2.
\end{align*}

Next we prove the reverse inequality. For any $\varepsilon>0$, there exists $\hat{\psi}_{\varepsilon}\in \mathcal{C}_\mathcal{S}$ such that 
\begin{align*}
\sum_{i=1}^n(\hat{\psi}_{\varepsilon}(X_i)-Y_i)^2\leq\min_{\psi\in\mathcal{C}_\mathcal{S}}
\sum_{i=1}^n(\psi(X_i)-Y_i)^2+\varepsilon.
\end{align*}
If we take $\hat{\xi}_{\varepsilon,i}\in \partial \hat{\psi}_{\varepsilon}(X_i)$, $i=1,\cdots,n$, then
\begin{align*}
(\hat{\psi}_{\varepsilon}(X_1),\cdots,\hat{\psi}_{\varepsilon}(X_n),\hat{\xi}_{\varepsilon,1},\cdots,\hat{\xi}_{\varepsilon,n})\in \mathcal{F}_\mathcal{S},\quad 
\hat{\phi}_{\varepsilon}(x):=\max_{1\leq j\leq n}\{  \hat{\psi}_{\varepsilon}(X_j)+\langle \hat{\xi}_{\varepsilon,j},x-X_j\rangle\}\in\mathcal{K}_\mathcal{S}.
\end{align*}
The inequalities $\hat{\psi}_{\varepsilon}(X_i) \geq \hat{\psi}_{\varepsilon}(X_j) + \langle \hat{\xi}_{\varepsilon,j},X_i - X_j\rangle$ for all $i,j$ implies that 
\begin{align*}
\hat{\phi}_{\varepsilon}(X_i) = \max_{1\leq j\leq n}\{  \hat{\psi}_{\varepsilon}(X_j)+\langle \hat{\xi}_{\varepsilon,j},X_i-X_j\rangle\}=\hat{\psi}_{\varepsilon}(X_i),\quad i = 1,\cdots,n.
\end{align*} 
Then, it holds that
\begin{align*}
\min_{\phi\in\mathcal{K}_\mathcal{S}}
\sum_{i=1}^n(\phi(X_i)-Y_i)^2\leq \sum_{i=1}^n(\hat{\phi}_{\varepsilon}(X_i)-Y_i)^2=\sum_{i=1}^n(\hat{\psi}_{\varepsilon}(X_i)-Y_i)^2\leq\min_{\psi\in\mathcal{C}_\mathcal{S}}
\sum_{i=1}^n(\psi(X_i)-Y_i)^2+\varepsilon.
\end{align*}
Since the above inequality holds for any $\varepsilon>0$, the equality \eqref{eq:twomin} follows. Now suppose $\hat{\phi}$ is an optimal solution to \eqref{min_KS}, since $\hat{\phi}\in \mathcal{C}_\mathcal{S}$, we have that $\hat{\phi}$ is a solution to the problem \eqref{eq:min_CS}.
\end{proof}

The theorem above provides a tractable approach to compute \eqref{eq:min_CS} through solving \eqref{min_KS}. 
By definition, any function $\phi$ in $\mathcal{K}_\mathcal{S}$, which is determined by $(\theta_1,\cdots,\theta_n,\xi_1,\cdots,\xi_n)\in \mathcal{F}_\mathcal{S}$, satisfies
\begin{align*}
\phi(X_i)=\max_{1\leq j\leq n}\{ \theta_j+\langle \xi_j,X_i-X_j\rangle\}=\theta_i,\quad i=1,\cdots,n.
\end{align*}
Therefore, we can conclude the mechanism for computing an optimal solution to \eqref{eq:min_CS} as follows.

\paragraph{The mechanism for shape-constrained convex regression.} Suppose $\{(\hat{\theta}_i,\hat{\xi}_i)\}_{i=1}^n$ is an optimal solution to
\begin{align}\label{convex_LSE_D}
\min_{\theta_1,\ldots,\theta_n\in\mathbb{R};  \xi_1,\ldots,\xi_n \in\mathbb{R}^{d}}  
\Big\{  \frac{1}{2}\|\theta-Y\|^2 \Bigm\lvert (\theta_1,\cdots,\theta_n,\xi_1,\cdots,\xi_n)\in \mathcal{F}_\mathcal{S}\Big\},
\end{align}
where the feasible set $\mathcal{F}_\mathcal{S}$ is defined as in \eqref{eq:FS}. We can construct an optimal solution to \eqref{eq:min_CS} by taking
\begin{align} 
\hat{\psi}(x) = \max_{1\leq j\leq n} \Big\{ \hat{\theta}_j + \langle \hat{\xi}_j,x- X_j\rangle 
\Big\}, \quad  x\in \Omega.\label{extension_shape}
\end{align}
As one can see, the main task in our mechanism for estimating the shape-constrained convex function is to solve the constrained convex quadratic programming problem \eqref{convex_LSE_D}. 

Define the mapping $\mathcal{A}:\mathbb{R}^n\rightarrow \mathbb{R}^{n\times n}$ as $\mathcal{A}z=ze_n^T-e_nz^T$. Then it holds that $\mathcal{A}^*Z = (Z-Z^T)e_n$ for $Z\in \mathbb{R}^{n\times n}$, and $\mathcal{A}^*\mathcal{A} = 2nI_n-2e_ne_n^T$. Denote $\xi=(\xi_1;\cdots;\xi_n)\in \mathbb{R}^{dn}$, and $B_i=e_nX_i^T-X^T\in \mathbb{R}^{n\times d}$, then define $\mathcal{B}:\mathbb{R}^{dn}\rightarrow \mathbb{R}^{n\times n}$ as $\mathcal{B}\xi=(B_1\xi_1,\cdots,B_n\xi_n)$ for $\xi\in \mathbb{R}^{dn}$. Therefore, we have that $\mathcal{B}^* Z=(Z_1^T B_1,\cdots,Z_n^T B_n)^T$, and $\mathcal{B}^*\mathcal{B}$ is a block diagonal matrix in $\mathbb{R}^{dn\times dn}$ whose $i$-th block is $B_i^TB_i$. Based on these notations, the problem \eqref{convex_LSE_D} can equivalently be written as
\begin{align}
\min_{\theta\in \mathbb{R}^n,\xi\in \mathbb{R}^{dn}} \Big\{\frac{1}{2}\|\theta-Y\|^2+p(\xi)+\delta_{+}(\mathcal{A}\theta+\mathcal{B}\xi)\Big\},\tag{P}\label{reformulated_problem}
\end{align}
where $p(\xi)= \sum_{i=1}^n \delta_\mathcal{D}(\xi_i)$ and $\delta_{\pm}(\cdot)$ is the indicator function of $\mathbb{R}^{n\times n}_{\pm}$.

\paragraph{Smoothing approximation.} The function $\hat{\psi}$ obtained by \eqref{extension_shape} is nonsmooth. When a smooth function is required, we then need to compute a smooth approximation to $\hat{\psi}$. The idea of Nesterov's smoothing \cite{nesterov2005smooth} could be applied, and the details is described in \cite[Section 3]{mazumder2019computational}. Alternatively, one can use the Moreau envelope as a smooth approximation of $\hat{\psi}$, namely
\begin{align}
\hat{\psi}^{\rm M}_{\tau}(x)=\tau {\rm E}_{\hat{\psi}/\tau}(x)=\min_{y\in \mathbb{R}^d} \Big\{\hat{\psi}(y)+\frac{\tau}{2}\|y-x\|^2
\Big\},\label{smooth_moreau}
\end{align}
where $\tau>0$ is a regularization parameter. Note that 
\begin{align*}
\hat{\psi}^{\rm M}_{\tau}(x) = \min_{y\in \mathbb{R}^d,t\in \mathbb{R}}\Big\{t+\frac{\tau}{2}\|y-x\|^2\Bigm\lvert
t\geq \langle \hat{\xi}_j,y\rangle -\langle \hat{\xi}_j,X_j\rangle+\hat{\theta}_j,\ j=1,\cdots,n
\Big\},
\end{align*}
and the unique solution ${\rm Prox}_{\hat{\psi}/\tau}(x)$ of \eqref{smooth_moreau} can be obtained by solving a quadratic programming of dimension $d+1$, which could be efficiently computed by {\tt Gurobi} or {\tt MOSEK}. One can see that for any $\tau>0$, $\hat{\psi}^{\rm M}_{\tau}$ is convex, and differentiable with $\nabla \hat{\psi}^{\rm M}_{\tau}(x)=\tau(x-{\rm Prox}_{\hat{\psi}/\tau}(x))$. In addition, according to \cite{beck2012smoothing}, the approximation $\hat{\psi}^{\rm M}_{\tau}$ of $\hat{\psi}$ satisfies the approximation bound
\begin{align*}
0\leq \hat{\psi}(x)-\hat{\psi}^{\rm M}_{\tau}(x)\leq \frac{1}{\tau}{\rm dist}^2(0,\partial \hat{\psi}(x))\leq \frac{L^2}{2\tau},\quad \forall x\in \Omega,
\end{align*}
where $L=\max\{\|\xi_j\|_2\mid j=1,\cdots,n\}$.

\paragraph{Dual problem and optimality conditions.} To derive its dual, it is convenient for us to write \eqref{reformulated_problem} as
\begin{align}
\min_{\theta\in \mathbb{R}^n,\xi,y\in \mathbb{R}^{dn},\eta\in \mathbb{R}^{n\times n}} \Big\{\frac{1}{2}\|\theta-Y\|^2+p(y)+\delta_{-}(\eta)\Bigm\lvert\eta+\mathcal{A}\theta+\mathcal{B}\xi=0,\ \xi-y=0\Big\}.\label{P}
\end{align}
The associated Lagrangian function is
\begin{align*}
l(\theta,\xi,y,\eta;u,v)=\frac{1}{2}\|\theta-Y\|^2+p(y)+\delta_{-}(\eta)-\langle u,\eta+\mathcal{A}\theta+\mathcal{B}\xi\rangle-\langle v,\xi-y\rangle.
\end{align*}
By minimizing $l(\theta,\xi,y,\eta;u,v)$ with respect to $\theta,\xi,y,\eta$, the dual problem of \eqref{reformulated_problem} is given by
\begin{align}
\max_{u\in\mathbb{R}^{n\times n},v\in \mathbb{R}^{dn}}\Big\{-\frac{1}{2}\|\mathcal{A}^*u\|^2-\langle Y,\mathcal{A}^*u\rangle-p^*(-v)-\delta_{+}(u)\Bigm\lvert \mathcal{B}^*u+v=0\Big\}.\tag{D}\label{D}
\end{align}
The Karush-Kuhn-Tucker (KKT) conditions associated with \eqref{reformulated_problem} and \eqref{D} are given as follows:
\begin{align}\label{KKT_system}
\theta-Y-\mathcal{A}^*u=0,\quad \mathcal{B}^*u+v=0,\ -v\in\partial p(\xi),\quad -u\in \partial \delta_{+}(\mathcal{A}\theta+\mathcal{B}\xi).
\end{align}

\section{Symmetric Gauss-Seidel based alternating direction method of multipliers ({\tt sGS-ADMM}) for \eqref{reformulated_problem}}\label{sec:sGSADMM}
The popular first-order alternating direction method of multipliers ({\tt ADMM}) can be applied to solve \eqref{reformulated_problem}. In \cite[Section A.2]{mazumder2019computational}, the problem \eqref{reformulated_problem} is reformulated as 
\begin{align*}
\min_{\theta\in \mathbb{R}^n,\xi\in \mathbb{R}^{dn},\eta\in \mathbb{R}^{n\times n}} \Big\{\frac{1}{2}\|\theta-Y\|^2+p(\xi)+\delta_{-}(\eta)\mid\eta+\mathcal{A}\theta+\mathcal{B}\xi=0\Big\}.
\end{align*}
The corresponding augmented Lagrangian function for a fixed $\sigma>0$ is defined by
\begin{align*}
\widetilde{\mathcal{L}}_{\sigma}(\theta,\xi,\eta;u)=\frac{1}{2}\|\theta-Y\|^2+p(\xi)+\delta_{-}(\eta)+\frac{\sigma}{2}\|\eta+\mathcal{A}\theta+\mathcal{B}\xi-\frac{u}{\sigma}\|^2-\frac{1}{2\sigma}\|u\|^2.
\end{align*}
Then the two-block {\tt ADMM} is given as
\begin{align*}
\left\{\begin{aligned}
&\xi^{k+1} = \arg\min \widetilde{\mathcal{L}}_{\sigma}(\theta^{k},\xi,\eta^{k};u^{k})=\arg\min\Big\{
p(\xi)+\frac{\sigma}{2}\|\eta^k+\mathcal{A}\theta^k+\mathcal{B}\xi-\frac{u^k}{\sigma}\|^2
\Big\},\\
&(\theta^{k+1},\eta^{k+1}) = \arg\min \widetilde{\mathcal{L}}_{\sigma}(\theta,\xi^{k+1},\eta;u^{k}),\\
& u^{k+1}=u^k-\tau\sigma(\eta^{k+1}+\mathcal{A}\theta^{k+1}+\mathcal{B}\xi^{k+1}),
\end{aligned}\right.
\end{align*}
where $\tau\in(0,(1+\sqrt{5}/2))$ is a given step length. As described in \cite{mazumder2019computational}, the subproblem of updating $\xi$ is separable in the variables $\xi_i$'s for $i=1,\cdots,n$, and the update of each $\xi_i$ can be solved by using an interior point method. The update of $\theta$ and $\eta$ is performed by using a block coordinate descent method, which may converge slowly. One can also apply the directly extended three-block {\tt ADMM} algorithm as in \cite[Section 2.1]{mazumder2019computational} to solve \eqref{reformulated_problem}, and the steps are given by
\begin{align*}
\left\{\begin{aligned}
&\xi^{k+1} = \arg\min \widetilde{\mathcal{L}}_{\sigma}(\theta^{k},\xi,\eta^{k};u^{k}),\\
&\theta^{k+1} = \arg\min \widetilde{\mathcal{L}}_{\sigma}(\theta,\xi^{k+1},\eta^k;u^{k}),\\
&\eta^{k+1} = \arg\min \widetilde{\mathcal{L}}_{\sigma}(\theta^{k+1},\xi^{k+1},\eta;u^{k}),\\
& u^{k+1}=u^k-\tau\sigma(\eta^{k+1}+\mathcal{A}\theta^{k+1}+\mathcal{B}\xi^{k+1}).
\end{aligned}\right.
\end{align*}
In the directly extended three-block {\tt ADMM}, the subproblem of updating $\theta$ can be computed by solving a linear system, and that of updating $\eta$ can be solved by the projection onto $\mathbb{R}_{-}^{n\times n}$. However, it is shown in \cite{chen2016direct} that the directly extended three-block {\tt ADMM} may not be convergent. 

In this section, we aim to present a convergent multi-block {\tt ADMM} for solving \eqref{reformulated_problem}. The authors in \cite{chen2017efficient} have proposed an inexact symmetric Gauss-Seidel based multi-block {\tt ADMM} for solving high-dimensional convex composite conic optimization problems, and it was demonstrated to perform better than the possibly nonconvergent directly extended multi-block {\tt ADMM}. Given a parameter $\sigma>0$, the augmented Lagrangian function associated with \eqref{P} is defined by
\begin{align}
&\mathcal{L}_{\sigma}(\theta,\xi,y,\eta;u,v)=l(\theta,\xi,y,\eta;u,v)+\frac{\sigma}{2}\|\eta+\mathcal{A}\theta+\mathcal{B}\xi\|^2+\frac{\sigma}{2}\|\xi-y\|^2\notag\\
&=\frac{1}{2}\|\theta-Y\|^2+p(y)+\delta_{-}(\eta)+\frac{\sigma}{2}\|\eta+\mathcal{A}\theta+\mathcal{B}\xi-\frac{u}{\sigma}\|^2+\frac{\sigma}{2}\|\xi-y-\frac{v}{\sigma}\|^2-\frac{1}{2\sigma}\|u\|^2-\frac{1}{2\sigma}\|v\|^2.
\label{augmented_lagrangian}
\end{align}
Then the symmetric Gauss-Seidel based {\tt ADMM} ({\tt sGS-ADMM}) algorithm for solving \eqref{reformulated_problem} is given as follows.

\begin{varalgorithm}{{\tt sGS-ADMM}}\small
	\caption{: Symmetric Gauss-Seidel based {\tt ADMM} for \eqref{reformulated_problem}} 
	\label{alg:sGSadmm}
	\begin{algorithmic}
		\STATE {\bfseries Initialization:} Choose an initial point $(\theta^0,\xi^0,y^0,\eta^0,u^0,v^0)\in \mathbb{R}^n\times \mathbb{R}^{dn}\times \mathbb{R}^{dn}\times\mathbb{R}^{n\times n}\times\mathbb{R}^{n\times n}\times \mathbb{R}^{dn}$, and parameter $\sigma > 0$. For $k = 0, 1, 2, \dots$
		\REPEAT
		\STATE {\bfseries Step 1}. Compute
		\begin{align*}
		(y^{k+1},\eta^{k+1}) = \arg\min \mathcal{L}_{\sigma}(\theta^{k},\xi^k,y,\eta;u^{k},v^{k}).
		\end{align*}
		\\
		\STATE {\bfseries Step 2}. Compute
		\begin{align*}
		&\mbox{{\bfseries Step 2a}.  } \widehat{\theta}^{k+1} = \arg\min  \mathcal{L}_{\sigma}(\theta,\xi^{k},y^{k+1},\eta^{k+1};u^{k},v^{k}),\\
		&\mbox{{\bfseries Step 2b}.  } \xi^{k+1} = \arg\min \mathcal{L}_{\sigma}(\widehat{\theta}^{k+1},\xi,y^{k+1},\eta^{k+1};u^{k},v^{k}),\\
		&\mbox{{\bfseries Step 2c}.  } \theta^{k+1} = \arg\min  \mathcal{L}_{\sigma}(\theta,\xi^{k+1},y^{k+1},\eta^{k+1};u^{k},v^{k}).
		\end{align*}
		\\
		\STATE {\bfseries Step 3}. Compute
		\begin{align*}
		u^{k+1}=u^k-\tau\sigma(\eta^{k+1}+\mathcal{A}\theta^{k+1}+\mathcal{B}\xi^{k+1}),\ v^{k+1}=v^k-\tau\sigma(\xi^{k+1}-y^{k+1}),
		\end{align*}
		where $\tau\in(0,(1+\sqrt{5})/2)$ is the step length that is typically chosen to be $1.618$. 
		\UNTIL{Stopping criterion is satisfied.}
	\end{algorithmic}
\end{varalgorithm}

In Algorithm \ref{alg:sGSadmm}, all the subproblems can be solved explicitly. In {\bf Step 1}, $\eta^{k+1}$ and $y^{k+1}$ are separable and can be solved independently as
\begin{align*}
y^{k+1} = {\rm Prox}_{p/\sigma}(\xi^k-v^k/\sigma),\quad \eta^{k+1} = \Pi_{-}(-\mathcal{A}\theta^k-\mathcal{B}\xi^k+u^k/\sigma),
\end{align*}
where $\Pi_{\pm}(\cdot)$ denotes the projection onto $\mathbb{R}^{n\times n}_{\pm}$. In {\bf Step 2a} and {\bf Step 2c}, $\theta$ can be computed by solving the following linear system
\begin{align*}
(I_n+\sigma \mathcal{A}^*\mathcal{A})\theta = Y-\sigma\mathcal{A}^* (\eta+\mathcal{B}\xi-u/\sigma).
\end{align*}
By noting that $\mathcal{A}^*\mathcal{A}=2nI_n-2e_ne_n^T$, one can apply the Sherman-Morrison-Woodbury formula to compute
\begin{align*}
(I_n+\sigma \mathcal{A}^*\mathcal{A})^{-1} = \frac{1}{1+2\sigma n}(I_n+2\sigma e_ne_n^T).
\end{align*}
Thus $\theta$ can be computed in $O(n)$ operations. For {\bf Step 2b}, $\xi^{k+1}$ can be computed by solving the linear equation
\begin{align*}
(I_{dn}+\mathcal{B}^*\mathcal{B})\xi = y^{k+1}+v^k/\sigma-\mathcal{B}^*(\eta^{k+1}+\mathcal{A}\widehat{\theta}^{k+1}-u^k/\sigma).
\end{align*}
As the coefficient matrix $I_{dn}+\mathcal{B}^*\mathcal{B}$ is a block diagonal matrix consisting $n$ blocks of $d\times d$ submatrices, each $\xi_i$ can be computed separately, and the inverse of each block only needs to be computed once. 

The convergence result of Algorithm \ref{alg:sGSadmm} is presented in the following theorem, which is taken directly from \cite[Theorem 5.1]{chen2017efficient}.
\begin{theorem}\label{thm:sgsadmm}
	Suppose that the solution set to the KKT system \eqref{KKT_system} is nonempty. Let $\{(\theta^k,\xi^k,y^k,\eta^k,u^k,v^k)\}$ be the sequence generated by Algorithm \ref{alg:sGSadmm}. Then the sequence $\{(\theta^k,\xi^k,y^k,\eta^k)\}$ converges to an optimal solution of problem \eqref{P}, and the sequence $\{(u^k,v^k)\}$ converges to an optimal solution of its dual \eqref{D}.
\end{theorem}

We can see that the {\tt sGS-ADMM} algorithm is easy to implement, but it is just a first-order algorithm which may not be efficient enough for solving \eqref{reformulated_problem} to high accuracy. In the next section, we design a superlinearly convergent proximal augmented Lagrangian method for solving \eqref{reformulated_problem}. By making full use of the special structure of the problem, we can exploit the second-order sparsity structure in the problem to greatly reduce the computational cost required in solving each of its subproblem.

\section{Proximal augmented Lagrangian method ({\tt pALM}) for \eqref{reformulated_problem}}\label{sec:pALM}
The augmented Lagrangian method is a desirable method for solving convex composite programming problems due to its superlinear convergence. To take advantage of the fast local convergence, we design a proximal augmented Lagrangian method ({\tt pALM}) for solving \eqref{reformulated_problem}. Note that the augmented Lagrangian function associated with \eqref{reformulated_problem} for any fixed $\sigma>0$ can be derived as
\begin{align*}
\widehat{\mathcal{L}}_{\sigma}(\theta,\xi;u,v)=
\inf_{y,\eta}\mathcal{L}_{\sigma}(\theta,\xi,y,\eta;u,v),
\end{align*}
where $\mathcal{L}_{\sigma}$ is defined in \eqref{augmented_lagrangian}. Therefore, by making use of the Moreau envelope,
\begin{align*}
\widehat{\mathcal{L}}_{\sigma}(\theta,\xi;u,v)=\frac{1}{2}\|\theta-Y\|^2+\sigma{\rm E}_{p}(\xi-\frac{v}{\sigma})+\sigma{\rm E}_{\delta_{-}}(-\mathcal{A}\theta-\mathcal{B}\xi+\frac{u}{\sigma}) -\frac{1}{2\sigma}\|u\|^2-\frac{1}{2\sigma}\|v\|^2.
\end{align*}
We propose the {\tt pALM} algorithm for solving \eqref{reformulated_problem} as follows.
 
\begin{varalgorithm}{{\tt pALM}}\small
	\caption{: Proximal augmented Lagrangian method for \eqref{reformulated_problem}}
	\label{alg:pALM}
	\begin{algorithmic}
		\STATE {\bfseries Initialization:} Let $H_1\in \mathbb{R}^{n\times n}$, $H_2\in \mathbb{R}^{dn\times dn}$ be given positive definite matrices, and $\{\varepsilon_k\}$ be a given summable sequence of nonnegative numbers. Choose a initial point $(\theta^0,\xi^0,u^0,v^0)\in \mathbb{R}^n\times \mathbb{R}^{dn}\times\mathbb{R}^{n\times n}\times \mathbb{R}^{dn}$, $\sigma_0 > 0$. For $k = 0, 1, 2, \dots$
		\REPEAT
		\STATE {\bfseries Step 1}. Compute
		\begin{align}
		&(\theta^{k+1},\xi^{k+1}) \approx \arg\min \Big\{\Phi_k(\theta,\xi):=\widehat{\mathcal{L}}_{\sigma}(\theta,\xi;u^k,v^k)+\frac{1}{2\sigma_k}\|\theta-\theta^k\|_{H_1}^2+\frac{1}{2\sigma_k}\|\xi-\xi^k\|_{H_2}^2\Big\},\label{eq:solve_thetaxi}
		\end{align}
	    such that the approximation solution $(\theta^{k+1},\xi^{k+1})$ satisfies the following stopping condition:
		\begin{align}
		\|\nabla \Phi_k(\theta^{k+1},\xi^{k+1})\|&\leq \frac{\lambda_{\min}}{\sigma_k}\varepsilon_k, \tag{A}\label{A}
		\end{align}
		where $\lambda_{\min} =\min\{\lambda_{\min}(H_1),\lambda_{\min}(H_2),1 \}$.
		\\[3pt]
		\STATE {\bfseries Step 2}. Update $u$, $v$ by
		\begin{align}
		&u^{k+1}=\sigma_k\Big[ u^k/\sigma^k-\mathcal{A}\theta^{k+1}-\mathcal{B}\xi^{k+1}-\Pi_{-}(u^k/\sigma^k-\mathcal{A}\theta^{k+1}-\mathcal{B}\xi^{k+1})\Big],\notag\\
		&v^{k+1}=-\sigma_k \Big[ \xi^{k+1}-v^k/\sigma_k-{\rm Prox}_p(\xi^{k+1}-v^k/\sigma^k) \Big].\notag
		\end{align}
		\\[3pt]
		\STATE {\bfseries Step 3}. Update $\sigma_{k+1} \uparrow \sigma_{\infty} \leq \infty$.
		\UNTIL{Stopping criterion is satisfied.}
	\end{algorithmic}
\end{varalgorithm}

\subsection{Convergence results for {\tt pALM}}\label{subsec:convergencepALM}
The Lagrangian function associated with \eqref{reformulated_problem} can be derived as
\begin{align*}
\widehat{l}(\theta,\xi,u,v)=\inf_{y,\eta}l(\theta,\xi,y,\eta;u,v)=\frac{1}{2}\|\theta-Y\|^2-p^*(-v)-\delta_{+}(u)-\langle u,\mathcal{A}\theta+\mathcal{B}\xi\rangle-\langle v,\xi\rangle.
\end{align*}
Define the maximal monotone operator $\mathcal{T}_{\widehat{l}}$ as
\begin{align*}
\mathcal{T}_{\widehat{l}}(\theta,\xi,u,v)=\Big\{(\theta',\xi',-u',-v')\mid(\theta',\xi',-u',-v')\in \partial \widehat{l}(\theta,\xi,u,v)
\Big\}.
\end{align*}
For any $(\bar{\theta},\bar{\xi},\bar{u},\bar{v})\in \mathbb{R}^n\times \mathbb{R}^{dn}\times \mathbb{R}^{n\times n}\times \mathbb{R}^{dn}$, denote
\begin{align*}
P_k(\bar{\theta},\bar{\xi},\bar{u},\bar{v})=\arg\min_{\theta,\xi}\max_{u,v}
\Big\{\widehat{l}(\theta,\xi;u,v)+\frac{1}{2\sigma_k}\|\theta-\bar{\theta}\|_{H_1}^2+\frac{1}{2\sigma_k}\|\xi-\bar{\xi}\|_{H_2}^2-\frac{1}{2\sigma_k}\|u-\bar{u}\|_F^2-\frac{1}{2\sigma_k}\|v-\bar{v}\|^2
\Big\}.
\end{align*}
Define the block diagonal operator $\Sigma = {\rm Diag}(H_1,H_2,\mathcal{I}_{n\times n},I_{dn})$, where $\mathcal{I}_{n\times n}$ is the identity operator of $\mathbb{R}^{n\times n}$. By mimicking the idea in \cite{li2019asymptotically}, one can prove the following proposition.
\begin{proposition}\label{prop:propertyTl}
(1) For any $k\geq 0$, it holds that for any $(\bar{\theta},\bar{\xi},\bar{u},\bar{v})\in \mathbb{R}^n\times \mathbb{R}^{dn}\times \mathbb{R}^{n\times n}\times \mathbb{R}^{dn}$,
\begin{align*}
P_k(\bar{\theta},\bar{\xi},\bar{u},\bar{v})=(\Sigma+\sigma_k \mathcal{T}_{\widehat{l}})^{-1}\Sigma (\bar{\theta},\bar{\xi},\bar{u},\bar{v}).
\end{align*}
If $(\theta^*,\xi^*,u^*,v^*)\in \mathcal{T}_{\widehat{l}}^{-1}(0)$, then $P_k(\theta^*,\xi^*,u^*,v^*)=(\theta^*,\xi^*,u^*,v^*)$.\\
(2) For all $k\geq 0$,
\begin{align*}
\|(\theta^{k+1},\xi^{k+1},u^{k+1},v^{k+1})-P_k(\theta^{k},\xi^{k},u^{k},v^{k})\|_{\Sigma}\leq 
\frac{\sigma_k}{\lambda_{\min}}\|\nabla \Phi_k(\theta^{k+1},\xi^{k+1})\|,
\end{align*}
where $\lambda_{\min} =\min\{\lambda_{\min}(H_1),\lambda_{\min}(H_2),1 \}$.
\end{proposition}

Based on the above proposition, we have the following convergence results, adapted from \cite{li2019asymptotically}, for Algorithm \ref{alg:pALM}.
\begin{theorem}\label{thm:convergence_pALM}
(1) Let $\{(\theta^{k},\xi^{k},u^{k},v^{k})\}$ be the sequence generated by Algorithm \ref{alg:pALM} with the stopping criterion \eqref{A}. Then $(\theta^{k},\xi^{k},u^{k},v^{k})$ is bounded,  $\{(\theta^{k},\xi^{k})\}$ converges to an optimal solution of \eqref{reformulated_problem}, and $\{(u^{k},v^{k})\}$ converges to an optimal solution of \eqref{D}.\\
(2) Let $\Lambda=\mathcal{T}_{\widehat{l}}^{-1}(0)$ be the primal-dual solution set. Let $r:=\sum_{i=0}^{\infty}\varepsilon_k+{\rm dist}_{\Sigma}((\theta^0,\xi^0,u^0,v^0),\Lambda)$. Assume that for this $r>0$, there exists a constant $\kappa>0$ such that $\mathcal{T}_{\widehat{l}}$ satisfies the following error bound assumption
\begin{align}
{\rm dist}((\theta,\xi,u,v),\Lambda)\leq \kappa {\rm dist}(0,\mathcal{T}_{\widehat{l}}(\theta,\xi,u,v)),\quad \mbox{$\forall (\theta,\xi,u,v)$ satisfying }{\rm dist}((\theta,\xi,u,v),\Lambda)\leq r. \label{error_bound}
\end{align}
Suppose that $\{(\theta^{k},\xi^{k},u^{k},v^{k})\}$ is the sequence generated by Algorithm \ref{alg:pALM} with the stopping criteria \eqref{A} and \eqref{B}, which is defined as
\begin{align}
\|\nabla \Phi_k(\theta^{k+1},\xi^{k+1})\|\leq \frac{\delta_k\lambda_{\min}}{\sigma_k}\|(\theta^{k+1},\xi^{k+1},u^{k+1},v^{k+1})-(\theta^{k},\xi^{k},u^{k},v^{k}) \|_{\Sigma},\tag{B}\label{B}
\end{align}
and $\{\delta_k\mid 0\leq \delta_k<1\}$ is a given summable sequence. Then it holds for all $k\geq 0$ that
\begin{equation}\label{eq:rate_pppa}
{\rm dist}_{\Sigma}((\theta^{k+1},\xi^{k+1},u^{k+1},v^{k+1}),\Lambda)\leq \mu_k {\rm dist}_{\Sigma}((\theta^{k},\xi^{k},u^{k},v^{k}),\Lambda),
\end{equation}
where
\[
\mu_k=\frac{1}{1-\delta_k}\frac{\delta_k+(1+\delta_k)\kappa \lambda_{\max} }{\sqrt{\sigma_k^2+\kappa^2\lambda_{\max}^2}}\rightarrow \mu_{\infty}=\frac{\kappa\lambda_{\max}}{\sqrt{
		\sigma_{\infty}^2+\kappa^2\lambda_{\max}^2}}<1,\quad k\rightarrow \infty,
\]
and $\lambda_{\max} = \max\{\lambda_{\max}(H_1),\lambda_{\max}(H_2),1\}$.
\end{theorem}

For specifying the convergence rate of Algorithm \ref{alg:pALM} for different choices of the closed convex set $\mathcal{D}$, we give the following remark.
\begin{remark}\label{remark:Tl}
	As one can see from Theorem \ref{thm:convergence_pALM}, the linear convergence rate of Algorithm \ref{alg:pALM} depends on the error bound assumption \eqref{error_bound} for the maximal monotone operator $\mathcal{T}_{\widehat{l}}$. It is well known that any polyhedral multifunction is upper Lipschitz continuous at every point of its domain according to \cite{robinson1981some}, which means it satisfies the error bound assumption \eqref{error_bound} for any $r>0$. For the cases when $\mathcal{D}$ is a polyhedral set, e.g. $\mathcal{D}=\mathbb{R}_{+}^d(\mathbb{R}_{-}^d)$ or $\mathcal{D}=\{x\in \mathbb{R}^d\mid \|x\|_q\leq L\}$ with $q=1$ or $q=\infty$, $\mathcal{T}_{\widehat{l}}$ is a polyhedral multifunction, and hence it satisfies the error bound assumption \eqref{error_bound}.
\end{remark}

\subsection{A semismooth Newton method for solving the {\tt pALM} subproblems}\label{subsec:ssn}
One can see that the most computationally intensive step in each of the {\tt pALM} is in solving the subproblem \eqref{eq:solve_thetaxi}. Here we described how it can be solved efficiently by the semismooth Newton method ({\tt SSN}). For given $\sigma>0$, $(\tilde{\theta},\tilde{\xi},\tilde{u},\tilde{v})\in \mathbb{R}^n\times \mathbb{R}^{dn}\times \mathbb{R}^{n\times n}\times \mathbb{R}^{dn}$, we aim to solve the {\tt pALM} subproblem, which has the form:
\begin{align}
\min \Big\{ \Phi(\theta,\xi):=\widehat{\mathcal{L}}_{\sigma}(\theta,\xi;\tilde{u},\tilde{v})+\frac{1}{2\sigma}\|\theta-\tilde{\theta}\|_{H_1}^2+\frac{1}{2\sigma}\|\xi-\tilde{\xi}\|_{H_2}^2\Big\}.\label{eq:pALMsub}
\end{align}
Since $\Phi(\cdot,\cdot)$ is strongly convex, the above minimization problem admits a unique solution $(\bar{\theta},\bar{\xi})$, which can be computed by solving the nonsmooth equation
\begin{align}\label{eq:nablapsi}
\nabla \Phi(\theta,\xi)=0,
\end{align}
where
\begin{align*}
\nabla \Phi(\theta,\xi)=\left(\begin{aligned}
&\theta-Y-\sigma\mathcal{A}^* \Big[(-\mathcal{A}\theta-\mathcal{B}\xi+\frac{\tilde{u}}{\sigma})-\Pi_{-}(-\mathcal{A}\theta-\mathcal{B}\xi+\frac{\tilde{u}}{\sigma})\Big]+\frac{1}{\sigma}H_1(\theta-\tilde{\theta})\\
&-\sigma \mathcal{B}^* \Big[(-\mathcal{A}\theta-\mathcal{B}\xi+\frac{\tilde{u}}{\sigma})-\Pi_{-}(-\mathcal{A}\theta-\mathcal{B}\xi+\frac{\tilde{u}}{\sigma})\Big]+\sigma\Big[(\xi-\frac{\tilde{v}}{\sigma})-{\rm Prox}_p(\xi-\frac{\tilde{v}}{\sigma})
\Big]+\frac{1}{\sigma}H_2(\xi-\tilde{\xi})
\end{aligned}\right).
\end{align*}
To apply the {\tt SSN} method to solve the above nonsmooth equation, we need a suitable generalized Jacobian of $\Phi(\cdot,\cdot)$, and we choose the following set as the candidate:
\begin{align*}
\hat{\partial}^2\Phi(\theta,\xi)=\sigma \left(\begin{aligned}
\mathcal{A}^*\\
\mathcal{B}^*
\end{aligned}\right)\Big[I-\partial \Pi_{-}(-\mathcal{A}\theta-\mathcal{B}\xi+\frac{\tilde{u}}{\sigma})\Big]\left(\begin{aligned}
&\mathcal{A} & \mathcal{B}
\end{aligned}\right)+\left(\begin{aligned}
I_n+&\frac{1}{\sigma}H_1\\
 &0
\end{aligned}\right.
\left.\begin{aligned}
&0\\
\sigma \Big[ I-\partial{\rm Prox}_p&(\xi-\frac{\tilde{v}}{\sigma})\Big]+\frac{1}{\sigma}H_2
\end{aligned}\right).
\end{align*}
With the suitably chosen generalized Jacobian, we can design the following {\tt SSN} method, which is a generalization of the standard Newton method, for solving \eqref{eq:pALMsub}.

\begin{varalgorithm}{{\tt SSN}}\small
	\caption{: Semismooth Newton method for \eqref{eq:pALMsub}}
	\label{alg:ssn}
	\begin{algorithmic}
		\STATE {\bfseries Initialization:} Given $(\theta^0,\xi^0)\in \mathbb{R}^n\times \mathbb{R}^{dn}$, $\bar{\gamma}\in (0, 1)$, $\tau \in (0, 1]$, $\delta \in (0, 1)$, and $\mu \in (0, 1/2)$. For $j = 0, 1, 2, \dots$
		\REPEAT
		\STATE {\bfseries Step 1}. Select an element $\mathcal{H}_j \in \hat{\partial}^{2} \Phi(\theta^j,\xi^j)$. Apply a direct method or the preconditioned conjugate gradient ({\tt PCG}) method to find an approximate solution $(\Delta \theta^j;\Delta \xi^j)\in \mathbb{R}^{n}\times \mathbb{R}^{dn}$ to
		\begin{equation}\label{eq: cg-system}
		\mathcal{H}_j(\Delta \theta;\Delta \xi) \approx - \nabla \Phi(\theta^j,\xi^j)
		\end{equation}
		such that $\|\mathcal{H}_j(\Delta \theta^j;\Delta \xi^j) + \nabla\Phi(\theta^j,\xi^j)\| \leq \min(\bar{\gamma}, \|\nabla\Phi(\theta^j,\xi^j)\|^{1+\tau})$.
		\\[3pt]
		\STATE {\bfseries Step 2}. Set $\alpha_j = \delta^{m_j}$, where $m_j$ is the smallest nonnegative integer $m$ for which
		$$\Phi(\theta^j + \delta^m \Delta \theta^j,\xi^j + \delta^m \Delta \xi^j) \leq \Phi(\theta^j,\xi^j) + \mu\delta^m \langle \nabla\Phi(\theta^j,\xi^j), (\Delta \theta^j;\Delta \xi^j)\rangle .$$
		\\[3pt]
		\STATE{\bfseries Step 3}. Set $\theta^{j+1} = \theta^j + \alpha_j \Delta \theta^j$, $\xi^{j+1}=\xi^j+\alpha_j \Delta \xi^j$.
		\UNTIL{Stopping criterion based on $\theta^{j+1}$ and $\xi^{j+1}$ is satisfied.}
	\end{algorithmic}
\end{varalgorithm}

The convergence analysis for Algorithm \ref{alg:ssn} can be established as in \cite{li2018highly}.
\begin{theorem}\label{thm:convergenceSSN}
    Suppose that ${\rm Prox}_{p} (\cdot)$ is strongly semismooth with respect to $\partial{\rm Prox}_{p} (\cdot)$. Let $\{(\theta^{j},\xi^{j})\}$ be the infinite sequence generated by Algorithm \ref{alg:ssn}. Then, $\{(\theta^{j},\xi^{j})\}$ converges to the unique optimal solution $(\bar{\theta},\bar{\xi})$ of problem \eqref{eq:pALMsub} and 
	\begin{align*}
	\|(\theta^{j+1},\xi^{j+1})-(\bar{\theta},\bar{\xi})\|=O(\|(\theta^{j},\xi^{j})-(\bar{\theta},\bar{\xi})\|^{1+\tau}).
	\end{align*}
\end{theorem}
\begin{proof}
	Due to the strong convexity of $\Phi(\cdot,\cdot)$, we can see that $(\theta^j,\xi^j)$ converges to the unique optimal solution $(\bar{\theta},\bar{\xi})$ \cite[Proposition 3.3 and Theorem 3.4]{zhao2010newton}. By the formulation of $\hat{\partial}^2\Phi(\cdot,\cdot)$, we have that all the elements in $\hat{\partial}^2\Phi(\theta,\xi)$ are positive definite for any $(\theta,\xi)\in\mathbb{R}^{n}\times\mathbb{R}^{dn}$ due to the positive definiteness of $H_1$ and $H_2$. Then the convergence rate of $(\theta^j,\xi^j)$ can be directly obtained from \cite[Theorem 3.5]{zhao2010newton}.
\end{proof}
\begin{remark}\label{remark:stronglysemismooth}
	As a side note, for the closed convex set $\mathcal{D}$ defined in Proposition \ref{prop:defD}, the assumption that ${\rm Prox}_{p} (\cdot)$ is strongly semismooth with respect to $\partial{\rm Prox}_{p} (\cdot)$ always holds.
\end{remark}

\section{Implementation of the algorithms}\label{sec:proximalmapping}
In this section, we discuss some numerical details concerning the efficient implementation of two proposed algorithms. For implementing the proposed algorithms, we need the proximal mapping ${\rm Prox}_p(\xi)$ for any $\xi\in\mathbb{R}^{dn}$ and its generalized Jacobian. In addition, when evaluating the function value of the dual problem \eqref{D}, we need the formula for $p^*(\cdot)$. 

\subsection{Computation associated with $\mathcal{D}$}\label{subsec:D}
For any $\xi=(\xi_1;\cdots;\xi_n)\in \mathbb{R}^{dn}$, since $p(\xi)= \sum_{i=1}^n \delta_\mathcal{D}(\xi_i)$, we have that
\begin{align}\label{eq:pfunction}
p^*(\xi)=\sum_{i=1}^n \delta_\mathcal{D}^*(\xi_i),\quad
{\rm Prox}_{ p}(\xi)=\left(\begin{array}{c}
\Pi_\mathcal{D}(\xi_1)\\
\vdots\\
\Pi_\mathcal{D}(\xi_n)
\end{array}\right),\quad 
\partial {\rm Prox}_{ p}(\xi)=\left(\begin{array}{ccc}
\partial \Pi_\mathcal{D}(\xi_1) & &\\
&\ddots&\\
&&\partial\Pi_\mathcal{D}(\xi_n)
\end{array}\right),
\end{align}
which means that we just need to focus on $\delta_\mathcal{D}^*(\cdot)$, $\Pi_\mathcal{D}(\cdot)$ and $\partial \Pi_\mathcal{D}(\cdot)$ for each of the $\mathcal{D}$'s defined in Proposition \ref{prop:defD}. We summarize the results in Table \ref{table_conjugate} - Table \ref{table_jacobian}, and the content on the generalized Jacobian follows the idea in \cite{han1997newton,li2017efficient}. The detailed derivation associated with the case when $\mathcal{D}=\{x\in \mathbb{R}^d\mid \|x\|_{1}\leq L\}$ can be found in the Appendix.

{\small
\begin{ThreePartTable}
	\renewcommand{\multirowsetup}{\centering}
	\begin{longtable}{cc}
		\caption{Conjugate function $\delta_\mathcal{D}^*(\cdot)$} \label{table_conjugate}\\[-6pt]
		\hline  
		\specialrule{0.01em}{2pt}{4pt}
		$\mathcal{D}$ &  $\delta_\mathcal{D}^*(x)$   \\
		\specialrule{0.01em}{2pt}{6pt}
		\endfirsthead
		
		\multicolumn{2}{c}{{ \tablename\ \thetable{} -- continued from previous page}} \\
		
		\specialrule{0.01em}{2pt}{4pt}
		$\mathcal{D}$ &  $\delta_\mathcal{D}^*(x)$   \\
		\specialrule{0.01em}{2pt}{6pt}
		\endhead
		
		\hline
		\multicolumn{2}{r}{{Continued on next page}} \\
		\hline \hline
		\endfoot
		
		\specialrule{0.01em}{4pt}{2pt} 
		\hline
		\endlastfoot	
		
		$\mathcal{D}=\{x\in \mathbb{R}^d\mid x_{K_1}\geq 0,x_{K_2}\leq 0\}$ & $\delta_\mathcal{D}^*(x)=\delta_{-}(x_{K_1})+\delta_{+}(x_{K_2})+\delta_{\{0\}}(x_{K_3})$, where $K_3:=\{1,\cdots,d\} \backslash (K_1 \cup K_2)$\\
		\specialrule{0.01em}{4pt}{4pt}
		$\mathcal{D}=\{x\in \mathbb{R}^d\mid L\leq x\leq U\}$ &  $\delta_\mathcal{D}^*(x)= \langle U,\max\{x,0\}\rangle + \langle L,\min\{x,0\}\rangle$\\
		\specialrule{0.01em}{4pt}{4pt}
		$\mathcal{D}=\{x\in \mathbb{R}^d\mid \|x\|_{q}\leq L\}$ & $\delta_\mathcal{D}^*(x)=L\|x\|_p,\quad 1/p+1/q=1$
	\end{longtable}
\end{ThreePartTable}
}

{\small
\begin{center}
	\begin{ThreePartTable}
		\begin{TableNotes}
			\item[\bfseries {\footnotesize Note:} ]  {\footnotesize $P_x={\rm Diag}({\rm sign}(x))\in \mathbb{R}^{d\times d}$, $\Pi_{\Delta_d}(\cdot)$ denotes the projection onto the simplex $\Delta_d=\{x\in \mathbb{R}^d\mid e_d^T x=1,x\geq 0\}$, which can be computed in $O(d\log (d))$ operations.}
		\end{TableNotes}
	\renewcommand{\multirowsetup}{\centering}
	\begin{longtable}{cc}
		\caption{Proximal mapping $\Pi_\mathcal{D}(\cdot)$} \label{table_proximalmapping}\\[-6pt]
		\hline 
		\specialrule{0.01em}{2pt}{4pt}
		$\mathcal{D}$ &  $\Pi_\mathcal{D}(\cdot)$   \\
		\specialrule{0.01em}{2pt}{6pt}
		\endfirsthead
		
		\multicolumn{2}{c}{{ \tablename\ \thetable{} -- continued from previous page}} \\
		
		\specialrule{0.01em}{2pt}{4pt}
		$\mathcal{D}$ &  $\Pi_\mathcal{D}(\cdot)$   \\
		\specialrule{0.01em}{2pt}{6pt}
		\endhead
		
		\hline
		\multicolumn{2}{r}{{Continued on next page}} \\
		\hline \hline
		\endfoot
		
		\specialrule{0.01em}{4pt}{2pt}
		\hline
		\insertTableNotes
		\endlastfoot
		
		$\mathcal{D}=\{x\in \mathbb{R}^d\mid x_{K_1}\geq 0,x_{K_2}\leq 0\}$ & $(\Pi_\mathcal{D}(x))_i=\left\{\begin{aligned}
		&0 && \mbox{if}\ i\in K_1, x_{i}< 0, \mbox{ or }i\in K_2,x_i>0\\
		&x_i && \mbox{otherwise}
		\end{aligned}\right.$ \\
		\specialrule{0.01em}{4pt}{4pt}
		$\mathcal{D}=\{x\in \mathbb{R}^d\mid L\leq x\leq U\}$ &  $(\Pi_\mathcal{D}(x))_i=\left\{\begin{aligned}
		&x_i && \mbox{if}\ L_i\leq x_i\leq U_i\\
		&0 && \mbox{otherwise}
		\end{aligned}\right.$\\
		\specialrule{0.01em}{4pt}{4pt}
		$\mathcal{D}=\{x\in \mathbb{R}^d\mid \|x\|_{\infty}\leq L\}$ & $(\Pi_\mathcal{D}(x))_i=\left\{\begin{aligned}
		&x_i && \mbox{if}\ |x_i|\leq L \\
		&{\rm sign}(x_i)L && \mbox{if}\ |x_i|> L
		\end{aligned}\right.$\\
		\specialrule{0.01em}{4pt}{4pt}
		$\mathcal{D}=\{x\in \mathbb{R}^d\mid \|x\|_{2}\leq L\}$ & $\Pi_\mathcal{D}(x)=\left\{\begin{aligned}
		&x && \mbox{if}\ \|x\|_2\leq L \\
		&L\frac{x}{\|x\|_2} && \mbox{otherwise}
		\end{aligned}\right.$\\
		\specialrule{0.01em}{4pt}{4pt}
		$\mathcal{D}=\{x\in \mathbb{R}^d\mid \|x\|_{1}\leq L\}$ & $\Pi_\mathcal{D}(x)=\left\{\begin{aligned}
		&x && \mbox{if}\ \|x\|_1\leq L \\
		&L P_x \Pi_{\Delta_d}(P_x x/L)&& \mbox{otherwise}
		\end{aligned}\right.$
	\end{longtable}
\end{ThreePartTable}
\end{center}

\begin{center}
	\begin{ThreePartTable}
		\begin{TableNotes}
			\item[\bfseries {\footnotesize Note:} ]  {\footnotesize $\widetilde{H}={\rm Diag}(r)-\frac{1}{{\rm nnz}(r)}rr^T\in \partial\Pi_{\Delta_d}(x)$, where $r\in \mathbb{R}^d$ is defined as $r_i=1$ if $\big(\Pi_{\Delta_d}(P_x x/L)\big)_i\neq 0 $, and $r_i=0$ otherwise.}
		\end{TableNotes}
		\renewcommand{\multirowsetup}{\centering}
		\begin{longtable}{cc}
			\caption{Generalized Jacobian of $\Pi_\mathcal{D}(\cdot)$} \label{table_jacobian}\\[-6pt]
			\hline 
			\specialrule{0.01em}{2pt}{4pt}
			$\mathcal{D}$ &  $\partial\Pi_\mathcal{D}(\cdot)$   \\
			\specialrule{0.01em}{2pt}{6pt}
			\endfirsthead

			\multicolumn{2}{c}{{ \tablename\ \thetable{} -- continued from previous page}}
			\\
			
			\hline
			\specialrule{0.01em}{2pt}{4pt}
			$\mathcal{D}$ &  $\partial\Pi_\mathcal{D}(\cdot)$   \\
			\specialrule{0.01em}{2pt}{6pt}
			\endhead
			
			\hline
			\multicolumn{2}{r}{{Continued on next page}}\\
			\hline \hline
			\endfoot
			
			\specialrule{0.01em}{4pt}{2pt}
			\hline
			\insertTableNotes
			\endlastfoot
			
			$\mathcal{D}=\{x\in \mathbb{R}^d\mid x_{K_1}\geq 0,x_{K_2}\leq 0\}$ & $\partial \Pi_\mathcal{D}(x)= {\rm Diag}(u),\ u_i\in \left\{\begin{aligned}
			&\{0\} && \mbox{if}\ i\in K_1, x_{i}< 0, \mbox{ or }i\in K_2,x_i>0 \\
			&[0,1] && \mbox{if}\ i\in K_1\cup K_2,x_i=0\\
			&\{1\} && \mbox{otherwise}
			\end{aligned}\right.$ \\
			\specialrule{0.01em}{4pt}{4pt}
			$\mathcal{D}=\{x\in \mathbb{R}^d\mid L\leq x\leq U\}$ &  $\partial \Pi_\mathcal{D}(x)= {\rm Diag}(u),\  u_i\in \left\{\begin{aligned}
			&\{1\} && \mbox{if}\ L_i<x_i<U_i \\
			&[0,1] && \mbox{if}\ x_i=L_i \mbox{ or }x_i=U_i \\[2pt]
			&\{0\} && \mbox{otherwise}
			\end{aligned}\right.$\\
			\specialrule{0.01em}{4pt}{4pt}
			$\mathcal{D}=\{x\in \mathbb{R}^d\mid \|x\|_{\infty}\leq L\}$ & $\partial\Pi_\mathcal{D}(x)={\rm Diag}(u),\ u_i\in\left\{\begin{aligned}
			&\{1\} && \mbox{if}\ |x_i|<L \\
			&[0,1] && \mbox{if}\ |x_i|=L \\
			&\{0\} && \mbox{otherwise}
			\end{aligned}\right.$\\
			\specialrule{0.01em}{4pt}{4pt}
			$\mathcal{D}=\{x\in \mathbb{R}^d\mid \|x\|_{2}\leq L\}$ & $\partial\Pi_\mathcal{D}(x)=\left\{\begin{aligned}
			&\{I_d\} && \mbox{if}\ \|x\|_2< L \\
			&\Big\{I_d-t\frac{xx^T}{L^2}\mid 0\leq t\leq 1\Big\} && \mbox{if}\ \|x\|_2= L\\
			&\Big\{\frac{L}{\|x\|_2}(I_d-\frac{xx^T}{\|x\|_2^2})\Big\} && \mbox{otherwise}
			\end{aligned}\right.$\\
			\specialrule{0.01em}{4pt}{4pt}
			$\mathcal{D}=\{x\in \mathbb{R}^d\mid \|x\|_{1}\leq L\}$ & $H\in \partial \Pi_\mathcal{D}(x),\ \mbox{where } H=\left\{\begin{aligned}
			&I_d && \mbox{if}\ \|x\|_1\leq L \\
			&P_x \widetilde{H} P_x && \mbox{otherwise}
			\end{aligned}\right.$
		\end{longtable}
	\end{ThreePartTable}
\end{center}
}

\subsection{Finding a computable element in $\hat{\partial}^2\Phi(\theta,\xi)$}\label{subsec:hessian}
The most difficult part of the {\tt pALM} algorithm is solving the Newton system \eqref{eq: cg-system}. For efficient practical implementation, we need to find an efficiently computable element in $\hat{\partial}^2\Phi(\theta,\xi)$ for any given $(\theta,\xi)\in\mathbb{R}^n\times \mathbb{R}^{dn}$. From the definition of $\hat{\partial}^2\Phi(\theta,\xi)$, we can rewrite it as
\begin{align*}
\hat{\partial}^2\Phi(\theta,\xi) = \mathcal{M}_1(\theta,\xi)+\mathcal{M}_2(\xi),
\end{align*}
where
\begin{align*}
\mathcal{M}_1(\theta,\xi) &= \sigma \left(\begin{aligned}
\mathcal{A}^*\\
\mathcal{B}^*
\end{aligned}\right)\Big[I-\partial \Pi_{-}(-\mathcal{A}\theta-\mathcal{B}\xi+\frac{\tilde{u}}{\sigma})\Big]\left(\begin{aligned}
&\mathcal{A} & \mathcal{B}
\end{aligned}\right),\\
\mathcal{M}_2(\xi) &= \left(\begin{aligned}
I_n+&\frac{1}{\sigma}H_1\\
&0
\end{aligned}\right.
\left.\begin{aligned}
&0\\
\sigma \Big[ I-\partial{\rm Prox}_p&(\xi-\frac{\tilde{v}}{\sigma})\Big]+\frac{1}{\sigma}H_2
\end{aligned}\right).
\end{align*}
Based on our discussion on $\partial {\rm Prox}_p(\cdot)$ in \eqref{eq:pfunction} and $\partial \Pi_\mathcal{D}(\cdot)$ in Table \ref{table_jacobian}, we can see that the elements in $\partial {\rm Prox}_p(\cdot)$ are block diagonal matrices. In order to maintain the block diagonal structure, we choose $H_1$ and $H_2$ to be diagonal matrices, and hence the elements in $\mathcal{M}_2(\xi)$ for any $\xi\in \mathbb{R}^{dn}$ will also be block diagonal matrices. One can easily pick an element in $\mathcal{M}_2(\xi)$ by choosing an element in $\partial{\rm Prox}_p(\xi-\tilde{v}/\sigma)$.
For $\mathcal{M}_1(\theta,\xi)$, to make full use of the second-order sparsity structure, we choose an element $W$ in $\partial \Pi_{-}(-\mathcal{A}\theta-\mathcal{B}\xi+\frac{\tilde{u}}{\sigma})$, where
\begin{align*}
W_{ij} = \left\{\begin{aligned}
&1 && \mbox{if}\ (-\mathcal{A}\theta-\mathcal{B}\xi+\frac{\tilde{u}}{\sigma})_{ij}\leq 0, \\
&0 && \mbox{otherwise}.
\end{aligned}\right.
\end{align*}
By denoting the $0$-$1$ matrix $I_{n}-W$ as $\bar{W}$, then
\begin{align*}
P:=\sigma \left(\begin{aligned}
\mathcal{A}^*\\
\mathcal{B}^*
\end{aligned}\right)\bar{W}\left(\begin{aligned}
&\mathcal{A} & \mathcal{B}
\end{aligned}\right) =\sigma\left(\begin{aligned}
&\mathcal{A}^* \bar{W}\mathcal{A} && \mathcal{A}^*\bar{W}\mathcal{B}\\
&\mathcal{B}^*\bar{W}\mathcal{A} && \mathcal{B}^* \bar{W}\mathcal{B}
\end{aligned}\right)
\end{align*}
is an element in $\mathcal{M}_1(\theta,\xi)$. After some algebraic manipulations by making use of the structure of $\mathcal{A}$ and $\mathcal{B}$, we can prove the following results:
\begin{align*}
\mathcal{A}^* \bar{W}\mathcal{A} &= {\rm Diag}(\bar{W}e_n)+{\rm Diag}(\bar{W}^T e_n)-\bar{W}-\bar{W}^T,\\
\mathcal{A}^*\bar{W}\mathcal{B}&=\Big({\rm Diag}(\bar{W}_1)B_1,\cdots,{\rm Diag}(\bar{W}_n)B_n \Big)-\left(\begin{array}{ccc}
\bar{W}_1^T B_1 & &\\
&\ddots&\\
&&\bar{W}_n^T B_n
\end{array}\right),\\
\mathcal{B}^* \bar{W}\mathcal{B}&=\left(\begin{array}{ccc}
B_1^T {\rm Diag}(\bar{W}_1)B_1& &\\
&\ddots&\\
&&B_n^T {\rm Diag}(\bar{W}_n)B_n
\end{array}\right).
\end{align*}
It can be seen that the $0$-$1$ structure of $\bar{W}$ will reduce many operations in matrix-matrix multiplications, and hence highly reduce the computational cost for $P$. For all $i$, ${\rm Diag}(\bar{W}_i)B_i$ is a matrix in $\mathbb{R}^{n\times d}$, with its $j$-th row being the $j$-th row of $B_i$ if $(\bar{W}_i)_j=1$, or the zero vector if $(\bar{W}_i)_j=0$. The computation of $\bar{W}_i^T B_i$ can be obtained by summing the non-zero rows of ${\rm Diag}(\bar{W}_i)B_i$, and the computation of $B_i^T {\rm Diag}(\bar{W}_i)B_i=({\rm Diag}(\bar{W}_i)B_i)^T({\rm Diag}(\bar{W}_i)B_i)$ can be highly reduced in the same way. 

The special structure of the elements in $\hat{\partial}^2\Phi(\theta,\xi)$ makes it possible for us to apply the second-order type {\tt pALM} algorithm for the huge constrained quadratic programming problem \eqref{reformulated_problem}, which contains $n(d+1)$ variables, $n(n-1)$ linear inequality constraints and $n$ possibly non-polyhedral constraints. 

\section{Numerical experiments}\label{sec:numerical}
In this section, we conduct some numerical experiments to demonstrate the performance of the {\tt sGS-ADMM} algorithm and the {\tt pALM} algorithm for solving \eqref{reformulated_problem}, under each case of $\mathcal{D}$ mentioned in Proposition \ref{prop:defD}. In addition, we design a data-driven Lipschitz estimation method to deal with the boundary effect of the convex regression problem. All our computational results are obtained by running MATLAB on a windows workstation (12-core, Intel Xeon E5-2680 @ 2.50GHz, 128 G RAM). 

\subsection{Computational performance of the algorithms for solving \eqref{reformulated_problem}}\label{sec:numericalcomparison}
In this subsection, we compare the performance of the {\tt sGS-ADMM} algorithm, the {\tt pALM} algorithm and {\tt MOSEK} for increasing $d$ and $n$. As for the three-block {\tt ADMM} proposed in \cite{mazumder2019computational}, we know that the {\tt sGS-ADMM} algorithm has been demonstrated to perform better than the possibly nonconvergent directly extended multi-block {\tt ADMM}. As we can see in \cite{aybat2014parallel}, as long as there is enough memory, {\tt MOSEK} performs quite a lot better than the parallel proximal gradient method ({\tt PAPG}). Since there is enough memory on our workstation, we just compare our algorithms with the state-of-the-art algorithm {\tt MOSEK}. 

\paragraph{Stopping criteria.}
We measure the infeasibilities for the primal and dual problems \eqref{reformulated_problem} and \eqref{D} by $R_{P},R_{D}$, and the complementary conditions by $R_{C}$, where
\begin{align*}
R_{\rm P} :=& \max\Big\{\frac{\|\xi-{\rm Prox}_p(\xi)\|}{1+\|\xi\|},\ \frac{\|\Pi_{-}(\mathcal{A}\theta+\mathcal{B}\xi)\|}{1+\|\mathcal{A}\theta\|+\|\mathcal{B}\xi\|}\Big\}, \\
R_{\rm D} :=& \max\Big\{\frac{\|\theta-Y-\mathcal{A}^*u\|}{1+\|Y\|+\|\theta\|+\|u\|},\ \frac{\|\mathcal{B}^*u+v\|}{1+\|u\|+\|v\|}\Big\}, \\
R_{\rm C} :=& \max\Big\{\frac{\|\xi-{\rm Prox}_p(\xi-v)\|}{1+\|\xi\|+\|v\|},\ 
\frac{\|\mathcal{A}\theta+\mathcal{B}\xi-\Pi_{+}(\mathcal{A}\theta+\mathcal{B}\xi-u)\|}{1+\|\mathcal{A}\theta\|+\|\mathcal{B}\xi\|+\|u\|}\Big\}.
\end{align*}
We stop the algorithm when
\begin{align*}
R_{\rm KKT}:=\max\{R_{\rm P},R_{\rm D},R_{\rm C}\}\leq \epsilon,
\end{align*}
where $\epsilon=10^{-6}$ is a given tolerance. In addition,  the algorithm will be stopped when it reaches the maximum computation time of $2$ hours or the pre-set maximum number of iterations ($200$ for {\tt pALM}, and $10000$ for {\tt sGS-ADMM}). 

\paragraph{Construction of synthetic datasets.}
For a given convex function $\psi:\mathbb{R}^d\rightarrow \mathbb{R}$, the synthetic dataset is generated via the procedure in \cite{mazumder2019computational}. We first generate $n$ samples $X_i\in \mathbb{R}^d$, $i=1,\cdots,n$ uniformly from $[-1,1]^d$, then the corresponding responses are given as
\begin{align*}
Y_i=\psi(X_i)+\varepsilon_i.
\end{align*}
The error $\varepsilon$ follows the normal distribution $\mathcal{N}(0,\sigma^2 I_n)$, where $\sigma^2 = {\rm Var}(\{\psi(X_i)\}_{i=1}^n)/{\rm SNR}$. In the experiments, we take ${\rm SNR}=3$.

\paragraph{Data preprocessing.}
Before we run the algorithms for the data $X=(X_1,\cdots,X_n)\in \mathbb{R}^{d\times n}$ and $Y\in\mathbb{R}^n$, we process the data so as to build a more predictive model. For the response $Y$ and each row of the predictor $X$, we mean center the vector and then standardize it to have unit $\ell_2$-norm.

\paragraph{Numerical results.}
The numerical results on the comparison among {\tt pALM}, {\tt sGS-ADMM} and {\tt MOSEK} can be found in Table \ref{table_org} - Table \ref{table_1norm}. We conduct experiments on the unconstrained convex regression problem and each case of shape-constrained convex regression we mentioned before. In Algorithm {\tt pALM}, we choose $H_1=10^{-3}I_{n}$ and $H_2=10^{-3}I_{dn}$. All the test functions are convex on $\mathbb{R}^d$ and satisfy some specified shape constraints. As one can see from the tables, both {\tt sGS-ADMM} and {\tt pALM} outperform the state-of-the-art solver {\tt MOSEK}. To be specified, when estimating the function $\psi(x)=\exp(p^T x)$ for the moderate $(d,n)=(100,1000)$, {\tt sGS-ADMM} is about $6$ times faster than {\tt MOSEK}, and {\tt pALM} is about $19$ times faster than {\tt MOSEK}. For the case when $d=200$, $n=3000$, which is a large problem with $603,000$ variables and about $9,000,000$ inequality constraints, {\tt MOSEK} runs out of memory, while {\tt pALM} could solve it within $3$ minutes and {\tt sGS-ADMM} takes $10$ minutes. From the tables, we can see that {\tt sGS-ADMM} performs much better than {\tt MOSEK} in each instance, and {\tt pALM} performs even better than {\tt sGS-ADMM}. In most of the cases, {\tt pALM} is at least $10$ times faster than {\tt MOSEK}.

{\small
	\begin{ThreePartTable}
		\begin{TableNotes}
			\item[\bfseries {\footnotesize Note:} ]  {\footnotesize ``16(20)" means ``{\tt pALM} iterations (total inner {\tt SSN} iterations)". O.M. means the algorithm runs out of memory. Time is in the format of hours:minutes:seconds.}
		\end{TableNotes}
		\renewcommand{\multirowsetup}{\centering}
		\begin{longtable}{ccccccccc}
			\caption{Convex regression for test function $\psi(x)=\exp(p^T x)$, where $p$ is a given random vector with each coordinate drawn
			from the standard normal distribution.} \label{table_org}\\[-6pt]
			\hline  
			\specialrule{0.01em}{1pt}{1pt}
			\multicolumn{2}{c}{
				\tikz{
					\node[below left, inner sep=-1pt] (def) {Algorithm};
					\node[above right,inner sep=0.5pt] (abc) {$(d,n)$};
					\draw (def.north west|-abc.north west) -- (def.south east-|abc.south east);
			}}
			& $(50,500)$ & $(50,1000)$ &  $(100,1000)$  & $(100,2000)$ & $(200,2000)$ & $(200,3000)$\\
			\specialrule{0.01em}{0pt}{2pt}
			\endfirsthead
			
			\multicolumn{8}{c}{{ \tablename\ \thetable{} -- continued from previous page}} \\
			
			\specialrule{0.01em}{1pt}{1pt}
			\multicolumn{2}{c}{
				\tikz{
					\node[below left, inner sep=-1pt] (def) {Algorithm};
					\node[above right,inner sep=0.5pt] (abc) {$(d,n)$};
					\draw (def.north west|-abc.north west) -- (def.south east-|abc.south east);
			}}
			& $(50,500)$ & $(50,1000)$ &  $(100,1000)$  & $(100,2000)$ & $(200,2000)$ & $(200,3000)$\\
			\specialrule{0.01em}{0pt}{2pt}
			\endhead
			
			\hline
			\multicolumn{8}{r}{{Continued on next page}} \\
			\hline \hline
			\endfoot
			
			\specialrule{0.01em}{0pt}{1pt} 
			\hline
			\insertTableNotes
			\endlastfoot	
			
			\multirow{3}*{\tabincell{c}{\tt pALM}} 
			& Iteration & 14(13) & 16(20) & 15(21) & 18(34) & 16(27) & 18(35)\\
			& Time & 00:00:02 & 00:00:06 & 00:00:10 & 00:01:05 & 00:01:01 & 00:02:49\\
			& $R_{\rm KKT}$ & 2.61e-9  & 1.43e-9 & 1.16e-8 & 1.91e-7 & 3.02e-7 & 4.89e-7 \\
			\specialrule{0.01em}{1pt}{1pt}
			\multirow{3}*{\tabincell{c}{\tt sGS-ADMM} } 	
			& Iteration & 389 & 562 & 411 & 719 & 447 & 563\\
			& Time & 00:00:08 & 00:00:38 & 00:00:31 & 00:04:47 & 00:03:26 & 00:09:39 \\
			& $R_{\rm KKT}$ & 9.95e-7  & 9.88e-7 & 9.90e-7 & 8.44e-7 & 9.91e-7 & 9.92e-7 \\
			\specialrule{0.01em}{1pt}{1pt}
			\multirow{3}*{\tabincell{c}{\tt MOSEK} } 	
			& Iteration & 10 & 11 & 11 & 10 & 10 & O.M.\\
			& Time & 00:00:19 & 00:01:47 & 00:03:10 & 00:18:50 & 00:37:37 & O.M. \\
			& $R_{\rm KKT}$ & 6.59e-9  & 3.92e-9 & 2.76e-10 & 1.01e-7 & 3.27e-9 & O.M. \\
		\end{longtable}
	\end{ThreePartTable}
}

{\small
	\begin{ThreePartTable}
		\renewcommand{\multirowsetup}{\centering}
		\begin{longtable}{ccccccccc}
			\caption{Convex regression with monotone constraint (non-decreasing) for the test function $\psi(x)=(e_d^T x)_{+}$.} \label{table_nondecreasing}\\[-6pt]
			\hline  
			\specialrule{0.01em}{1pt}{1pt}
			\multicolumn{2}{c}{
				\tikz{
					\node[below left, inner sep=-1pt] (def) {Algorithm};
					\node[above right,inner sep=0.5pt] (abc) {$(d,n)$};
					\draw (def.north west|-abc.north west) -- (def.south east-|abc.south east);
			}}
			& $(50,500)$ & $(50,1000)$ &  $(100,1000)$  & $(100,2000)$ & $(200,2000)$ & $(200,3000)$\\
			\specialrule{0.01em}{0pt}{2pt}
			\endfirsthead
			
			\multicolumn{8}{c}{{ \tablename\ \thetable{} -- continued from previous page}} \\
			
			\specialrule{0.01em}{1pt}{1pt}
			\multicolumn{2}{c}{
				\tikz{
					\node[below left, inner sep=-1pt] (def) {Algorithm};
					\node[above right,inner sep=0.5pt] (abc) {$(d,n)$};
					\draw (def.north west|-abc.north west) -- (def.south east-|abc.south east);
			}}
			& $(50,500)$ & $(50,1000)$ &  $(100,1000)$  & $(100,2000)$ & $(200,2000)$ & $(200,3000)$\\
			\specialrule{0.01em}{0pt}{2pt}
			\endhead
			
			\hline
			\multicolumn{8}{r}{{Continued on next page}} \\
			\hline \hline
			\endfoot
			
			\specialrule{0.01em}{0pt}{1pt} 
			\hline
			\endlastfoot	
			
			\multirow{3}*{\tabincell{c}{\tt pALM}} 
			& Iteration & 17(19) & 17(23) & 17(24) & 20(63) & 18(41) & 20(56)\\
			& Time & 00:00:02 & 00:00:08 & 00:00:13 & 00:02:48 & 00:01:59 & 00:05:49\\
			& $R_{\rm KKT}$ & 9.66e-9  & 1.50e-9 & 8.24e-8 & 2.90e-7 & 2.35e-7 & 7.69e-7 \\
			\specialrule{0.01em}{1pt}{1pt}
			\multirow{3}*{\tabincell{c}{\tt sGS-ADMM} } 	
			& Iteration & 529 & 917 & 481 & 960 & 541 & 716\\
			& Time & 00:00:12 & 00:01:07 & 00:00:39 & 00:06:46 & 00:04:28 & 00:13:04 \\
			& $R_{\rm KKT}$ & 8.04e-7  & 9.96e-7 & 9.92e-7 & 7.87e-7 & 9.19e-7 & 9.88e-7 \\
			\specialrule{0.01em}{1pt}{1pt}
			\multirow{3}*{\tabincell{c}{\tt MOSEK} } 	
			& Iteration & 14 & 13 & 13 & 17 & 12 & O.M.\\
			& Time & 00:00:24 & 00:01:58 & 00:03:34 & 00:25:28 & 00:43:34 & O.M. \\
			& $R_{\rm KKT}$ & 1.54e-9  & 1.45e-9 & 4.65e-7 & 8.99e-11 & 4.35e-7 & O.M. \\
		\end{longtable}
	\end{ThreePartTable}
}

{\small
	\begin{ThreePartTable}
		\renewcommand{\multirowsetup}{\centering}
		\begin{longtable}{ccccccccc}
			\caption{Convex regression with box constraint ($L=0_d$, $U=e_d$) for the test function $\psi(x)=\ln(1+\exp(e_d^Tx))$.} \label{table_twobound}\\[-6pt]
			\hline  
			\specialrule{0.01em}{1pt}{1pt}
			\multicolumn{2}{c}{
				\tikz{
					\node[below left, inner sep=-1pt] (def) {Algorithm};
					\node[above right,inner sep=0.5pt] (abc) {$(d,n)$};
					\draw (def.north west|-abc.north west) -- (def.south east-|abc.south east);
			}}
			& $(50,500)$ & $(50,1000)$ &  $(100,1000)$  & $(100,2000)$ & $(200,2000)$ & $(200,3000)$\\
			\specialrule{0.01em}{0pt}{2pt}
			\endfirsthead
			
			\multicolumn{8}{c}{{ \tablename\ \thetable{} -- continued from previous page}} \\
			
			\specialrule{0.01em}{1pt}{1pt}
			\multicolumn{2}{c}{
				\tikz{
					\node[below left, inner sep=-1pt] (def) {Algorithm};
					\node[above right,inner sep=0.5pt] (abc) {$(d,n)$};
					\draw (def.north west|-abc.north west) -- (def.south east-|abc.south east);
			}}
			& $(50,500)$ & $(50,1000)$ &  $(100,1000)$  & $(100,2000)$ & $(200,2000)$ & $(200,3000)$\\
			\specialrule{0.01em}{0pt}{2pt}
			\endhead
			
			\hline
			\multicolumn{8}{r}{{Continued on next page}} \\
			\hline \hline
			\endfoot
			
			\specialrule{0.01em}{0pt}{1pt} 
			\hline
			\endlastfoot	
			
			\multirow{3}*{\tabincell{c}{\tt pALM}} 
			& Iteration & 25(42) & 24(65) & 18(25) & 19(55) & 17(37) & 20(57)\\
			& Time & 00:00:04 & 00:00:23 & 00:00:14 & 00:02:32 & 00:01:50 & 00:06:03\\
			& $R_{\rm KKT}$ & 1.73e-7  & 1.63e-7 & 3.67e-8 & 4.77e-7 & 1.36e-7 & 7.39e-7 \\
			\specialrule{0.01em}{1pt}{1pt}
			\multirow{3}*{\tabincell{c}{\tt sGS-ADMM} } 	
			& Iteration & 663 & 1016 & 473 & 935 & 546 & 715\\
			& Time & 00:00:16 & 00:01:17 & 00:00:41 & 00:06:43 & 00:04:33 & 00:13:04 \\
			& $R_{\rm KKT}$ & 9.60e-7  & 7.47e-7 & 7.46e-7 & 8.83e-7 & 9.69e-7 & 9.95e-7 \\
			\specialrule{0.01em}{1pt}{1pt}
			\multirow{3}*{\tabincell{c}{\tt MOSEK} } 	
			& Iteration & 19 & 24 & 15 & 16 & 16 & O.M.\\
			& Time & 00:00:30 & 00:02:54 & 00:04:15 & 00:26:11 & 00:57:01 & O.M. \\
			& $R_{\rm KKT}$ & 9.95e-7  & 6.03e-8 & 2.06e-8 & 3.20e-10 & 3.94e-10 & O.M. \\
		\end{longtable}
	\end{ThreePartTable}
}

{\small
	\begin{ThreePartTable}
		\renewcommand{\multirowsetup}{\centering}
		\begin{longtable}{ccccccccc}
			\caption{Convex regression with Lipschitz constraint ($p=1$, $q=\infty$, $L=1$) for the test function $\psi(x)=\sqrt{1+x^Tx}$.} \label{table_infnorm}\\[-6pt]
			\hline  
			\specialrule{0.01em}{1pt}{1pt}
			\multicolumn{2}{c}{
				\tikz{
					\node[below left, inner sep=-1pt] (def) {Algorithm};
					\node[above right,inner sep=0.5pt] (abc) {$(d,n)$};
					\draw (def.north west|-abc.north west) -- (def.south east-|abc.south east);
			}}
			& $(50,500)$ & $(50,1000)$ &  $(100,1000)$  & $(100,2000)$ & $(200,2000)$ & $(200,3000)$\\
			\specialrule{0.01em}{0pt}{2pt}
			\endfirsthead
			
			\multicolumn{8}{c}{{ \tablename\ \thetable{} -- continued from previous page}} \\
			
			\specialrule{0.01em}{1pt}{1pt}
			\multicolumn{2}{c}{
				\tikz{
					\node[below left, inner sep=-1pt] (def) {Algorithm};
					\node[above right,inner sep=0.5pt] (abc) {$(d,n)$};
					\draw (def.north west|-abc.north west) -- (def.south east-|abc.south east);
			}}
			& $(50,500)$ & $(50,1000)$ &  $(100,1000)$  & $(100,2000)$ & $(200,2000)$ & $(200,3000)$\\
			\specialrule{0.01em}{0pt}{2pt}
			\endhead
			
			\hline
			\multicolumn{8}{r}{{Continued on next page}} \\
			\hline \hline
			\endfoot
			
			\specialrule{0.01em}{0pt}{1pt} 
			\hline
			\endlastfoot	
			
			\multirow{3}*{\tabincell{c}{\tt pALM}} 
			& Iteration & 15(16) & 17(30) & 15(20) & 18(38) & 18(36) & 21(43)\\
			& Time & 00:00:02 & 00:00:10 & 00:00:11 & 00:01:24 & 00:01:36 & 00:04:00\\
			& $R_{\rm KKT}$ & 1.01e-11  & 5.67e-8 & 1.52e-8 & 6.70e-7 & 4.05e-7 & 8.93e-7 \\
			\specialrule{0.01em}{1pt}{1pt}
			\multirow{3}*{\tabincell{c}{\tt sGS-ADMM} } 	
			& Iteration & 531 & 928 & 500 & 921 & 538 & 748\\
			& Time & 00:00:13 & 00:01:11 & 00:00:44 & 00:06:46 & 00:04:26 & 00:15:37 \\
			& $R_{\rm KKT}$ & 9.63e-7  & 9.97e-7 & 9.27e-7 & 9.97e-7 & 9.04e-7 & 8.63e-7 \\
			\specialrule{0.01em}{1pt}{1pt}
			\multirow{3}*{\tabincell{c}{\tt MOSEK} } 	
			& Iteration & 10 & 11 & 10 & 11 & 11 & O.M.\\
			& Time & 00:00:22 & 00:01:55 & 00:03:30 & 00:21:23 & 00:46:15 & O.M. \\
			& $R_{\rm KKT}$ & 7.51e-9  & 3.46e-10 & 3.81e-13 & 5.15e-10 & 2.86e-15 & O.M. \\
		\end{longtable}
	\end{ThreePartTable}
}

{\small
	\begin{ThreePartTable}
		\begin{TableNotes}
			\item[\bfseries {\footnotesize Note:} ]  {\footnotesize $Q\in \mathbb{R}^{d\times d}$ is a randomly generated positive definite matrix with known largest eigenvalue.}
		\end{TableNotes}
		\renewcommand{\multirowsetup}{\centering}
		\begin{longtable}{ccccccccc}
			\caption{Convex regression with Lipschitz constraint ($p=2$, $q=2$, $L=\lambda_{\rm max}(Q)$) for the test function $\psi(x)=\sqrt{x^T Qx}$.} \label{table_2norm}\\[-6pt]
			\hline  
			\specialrule{0.01em}{1pt}{1pt}
			\multicolumn{2}{c}{
				\tikz{
					\node[below left, inner sep=-1pt] (def) {Algorithm};
					\node[above right,inner sep=0.5pt] (abc) {$(d,n)$};
					\draw (def.north west|-abc.north west) -- (def.south east-|abc.south east);
			}}
			& $(50,500)$ & $(50,1000)$ &  $(100,1000)$  & $(100,2000)$ & $(200,2000)$ & $(200,3000)$\\
			\specialrule{0.01em}{0pt}{2pt}
			\endfirsthead
			
			\multicolumn{8}{c}{{ \tablename\ \thetable{} -- continued from previous page}} \\
			
			\specialrule{0.01em}{1pt}{1pt}
			\multicolumn{2}{c}{
				\tikz{
					\node[below left, inner sep=-1pt] (def) {Algorithm};
					\node[above right,inner sep=0.5pt] (abc) {$(d,n)$};
					\draw (def.north west|-abc.north west) -- (def.south east-|abc.south east);
			}}
			& $(50,500)$ & $(50,1000)$ &  $(100,1000)$  & $(100,2000)$ & $(200,2000)$ & $(200,3000)$\\
			\specialrule{0.01em}{0pt}{2pt}
			\endhead
			
			\hline
			\multicolumn{8}{r}{{Continued on next page}} \\
			\hline \hline
			\endfoot
			
			\specialrule{0.01em}{0pt}{1pt} 
			\hline
			\insertTableNotes
			\endlastfoot	
			
			\multirow{3}*{\tabincell{c}{\tt pALM}} 
			& Iteration & 13(12) & 17(30) & 15(20) & 19(42) & 17(35) & 32(59)\\
			& Time & 00:00:02 & 00:00:09 & 00:00:10 & 00:01:39 & 00:01:31 & 00:04:21\\
			& $R_{\rm KKT}$ & 2.05e-8  & 3.52e-8 & 7.52e-9 & 2.49e-7 & 4.84e-7 & 9.95e-7 \\
			\specialrule{0.01em}{1pt}{1pt}
			\multirow{3}*{\tabincell{c}{\tt sGS-ADMM} } 	
			& Iteration & 541 & 953 & 549 & 1005 & 499 & 705\\
			& Time & 00:00:14 & 00:01:19 & 00:00:53 & 00:07:29 & 00:04:18 & 00:13:38 \\
			& $R_{\rm KKT}$ & 9.35e-7  & 9.64e-7 & 7.53e-7 & 8.75e-7 & 8.86e-7 & 9.63e-7 \\
			\specialrule{0.01em}{1pt}{1pt}
			\multirow{3}*{\tabincell{c}{\tt MOSEK} } 	
			& Iteration & 10 & 13 & 11 & 12 & 11 & O.M.\\
			& Time & 00:00:22 & 00:02:04 & 00:03:41 & 00:22:20 & 00:46:19 & O.M. \\
			& $R_{\rm KKT}$ & 2.50e-7  & 1.06e-9 & 4.21e-9 & 3.20e-9 & 1.21e-8 & O.M. \\
		\end{longtable}
	\end{ThreePartTable}
}

{\small
	\begin{ThreePartTable}
		\renewcommand{\multirowsetup}{\centering}
		\begin{longtable}{ccccccccc}
			\caption{Convex regression with Lipschitz constraint ($p=\infty$, $q=1$, $L=1$) for the test function $\psi(x)=\ln(1+e^{x_1}+\cdots+e^{x_d})$.} \label{table_1norm}\\[-6pt]
			\hline  
			\specialrule{0.01em}{1pt}{1pt}
			\multicolumn{2}{c}{
				\tikz{
					\node[below left, inner sep=-1pt] (def) {Algorithm};
					\node[above right,inner sep=0.5pt] (abc) {$(d,n)$};
					\draw (def.north west|-abc.north west) -- (def.south east-|abc.south east);
			}}
			& $(50,500)$ & $(50,1000)$ &  $(100,1000)$  & $(100,2000)$ & $(200,2000)$ & $(200,3000)$\\
			\specialrule{0.01em}{0pt}{2pt}
			\endfirsthead
			
			\multicolumn{8}{c}{{ \tablename\ \thetable{} -- continued from previous page}} \\
			
			\specialrule{0.01em}{1pt}{1pt}
			\multicolumn{2}{c}{
				\tikz{
					\node[below left, inner sep=-1pt] (def) {Algorithm};
					\node[above right,inner sep=0.5pt] (abc) {$(d,n)$};
					\draw (def.north west|-abc.north west) -- (def.south east-|abc.south east);
			}}
			& $(50,500)$ & $(50,1000)$ &  $(100,1000)$  & $(100,2000)$ & $(200,2000)$ & $(200,3000)$\\
			\specialrule{0.01em}{0pt}{2pt}
			\endhead
			
			\hline
			\multicolumn{8}{r}{{Continued on next page}} \\
			\hline \hline
			\endfoot
			
			\specialrule{0.01em}{0pt}{1pt} 
			\hline
			\endlastfoot	
			
			\multirow{3}*{\tabincell{c}{\tt pALM}} 
			& Iteration & 14(14) & 16(22) & 15(20) & 18(40) & 17(36) & 21(41)\\
			& Time & 00:00:02 & 00:00:07 & 00:00:10 & 00:01:23 & 00:01:32 & 00:03:37\\
			& $R_{\rm KKT}$ & 7.99e-12  & 1.29e-8 & 2.13e-8 & 2.25e-7 & 3.85e-7 & 9.78e-7 \\
			\specialrule{0.01em}{1pt}{1pt}
			\multirow{3}*{\tabincell{c}{\tt sGS-ADMM} } 	
			& Iteration & 413 & 767 & 471 & 813 & 462 & 667\\
			& Time & 00:00:11 & 00:01:01 & 00:00:42 & 00:05:55 & 00:03:59 & 00:12:59 \\
			& $R_{\rm KKT}$ & 9.09e-7  & 9.96e-7 & 7.16e-7 & 9.94e-7 & 9.89e-7 & 9.94e-7 \\
			\specialrule{0.01em}{1pt}{1pt}
			\multirow{3}*{\tabincell{c}{\tt MOSEK} } 	
			& Iteration & 12 & 12 & 13 & 8 & 8 & O.M.\\
			& Time & 00:00:40 & 00:03:47 & 00:06:58 & 00:37:26 & 01:02:19 & O.M. \\
			& $R_{\rm KKT}$ & 1.23e-8  & 1.25e-7 & 2.97e-10 & 3.04e-9 & 4.52e-9 & O.M. \\
		\end{longtable}
	\end{ThreePartTable}
}

\subsection{Data-driven Lipschitz estimation method}\label{subsec:Lipschitz}
An important issue in convex regression is over-fitting near the boundary of ${\rm conv}(X_1,\cdots,X_n)$. That is to say, the norms of the fitted subgradients $\xi_i$'s near the boundary can become arbitrarily large. The authors in \cite{lim2014convergence,balazs2015near,mazumder2019computational} used the idea of Lipschitz convex regression to deal with this problem. They propose to compute the least squares estimator over the class of convex functions that are uniformly Lipschitz with a given bound, which means that they compute the estimator defined in \eqref{eq:min_CS} with Property $\mathcal{S}$ taking the form of (S3). In practice, the challenge is in choosing the unknown Lipschitz constant in the model based on the given data. Mazumder et al. \cite{mazumder2019computational} choose to estimate the Lipschitz constant by using the cross-validation. In this paper, we provide a data-driven Lipschitz estimation method for the Lipschitz convex regression. 

For each $X_i$, we first find the $k$-nearest neighbors $\mathcal{N}(X_i)$ of $X_i$, and then define 
\begin{align*}
L_i={\rm median}\Big\{\frac{|Y_i-Y_j|}{\|X_i-X_j\|_p},j\in \mathcal{N}(X_i)\Big\},
\end{align*}
where $p=1,2,\infty$ is given. Then we solve the generalization form of \eqref{convex_LSE_D} as
\begin{align}
&\min_{\theta_1,\ldots,\theta_n\in\mathbb{R};  \xi_1,\ldots,\xi_n \in\mathbb{R}^{d}} 
\frac{1}{2} \sum_{i=1}^n (\theta_i - Y_i)^2 \label{generalization_convexLSE}\\
&{\rm s.t.} \quad \theta_i \geq \theta_j + \langle \xi_j,X_i - X_j\rangle,\quad \forall\ 1\leq i ,j \leq n,\notag\\
&\qquad\ \xi_i\in \mathcal{D}_i,\quad i=1,\cdots,n,\notag
\end{align}
where $\mathcal{D}_i=\{x\in \mathbb{R}^d\mid \|x\|_q\leq L_i\}$ with $1/p+1/q=1$. The proposed {\tt sGS-ADMM} algorithm and {\tt pALM} algorithm can be easily extended to solve \eqref{generalization_convexLSE} by letting $p(\xi)= \sum_{i=1}^n \delta_{\mathcal{D}_i} (\xi_i)$.

We use an example here to demonstrate the performance of the Lipschitz convex regression with the data-driven Lipschitz estimation method. Consider the convex function $\psi(x)=2\|x\|_{\infty}+\|x\|^2$, we sample $n=80$ data points uniformly from $[-1,1]^d$ and add the Gaussian noise as stated in Section \ref{sec:numericalcomparison}. The results for $d=1,2$ can be seen in Figure \ref{fig:data_Lip}. When estimating the Lipschitz constant for each data point, we take $k=5$ and $p=q=2$. As shown in the figure, Lipschitz convex regression does benefit the performance of the regression near the boundary of the convex hull of $X_i$'s.

\begin{figure}[H]
	\subfigure[$d=1$]{\label{fig:toy1}
		\includegraphics[width=0.5\linewidth]{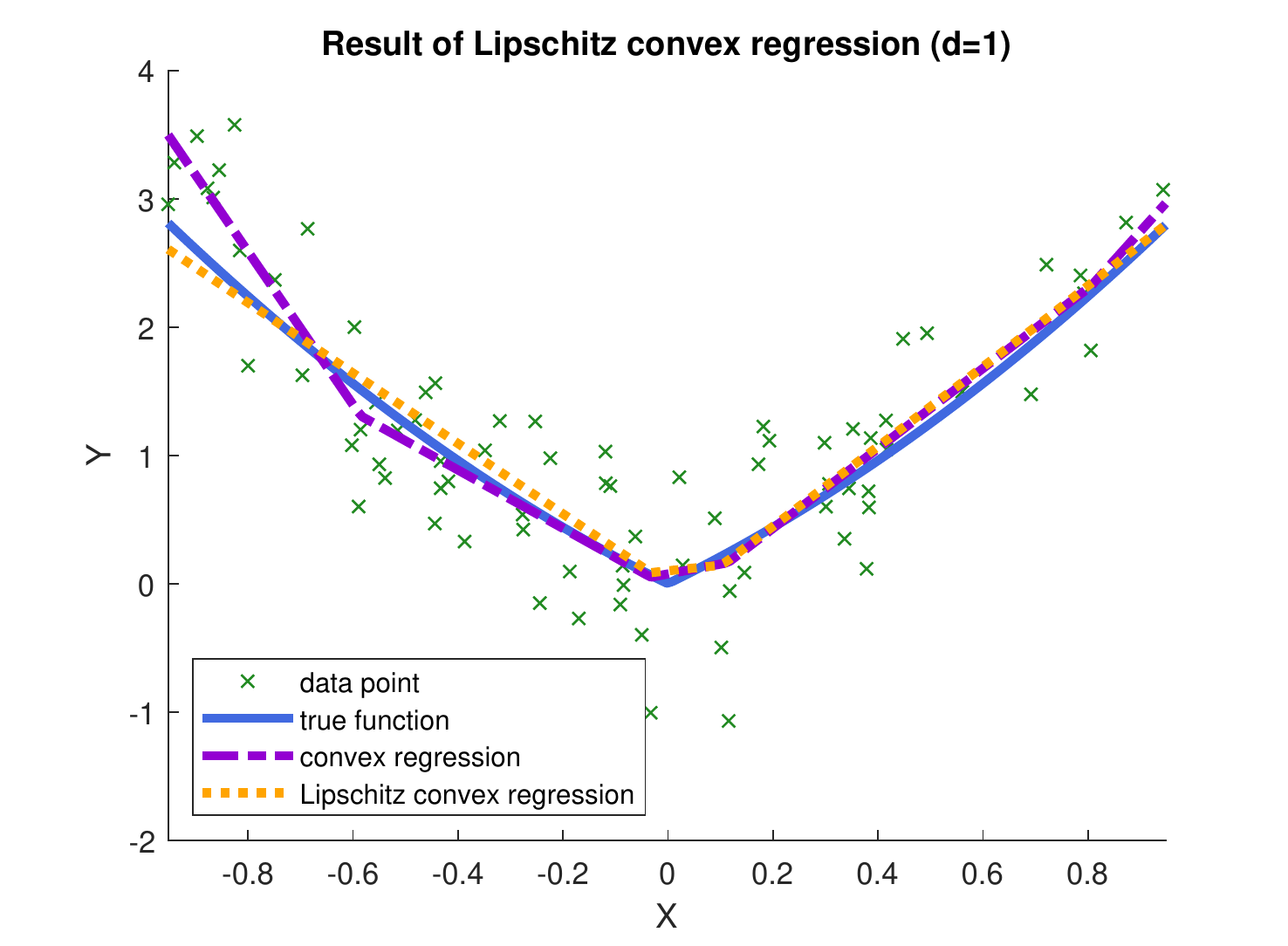}}
	\subfigure[$d=2$]{\label{fig:toy2}
		\includegraphics[width=0.5\linewidth]{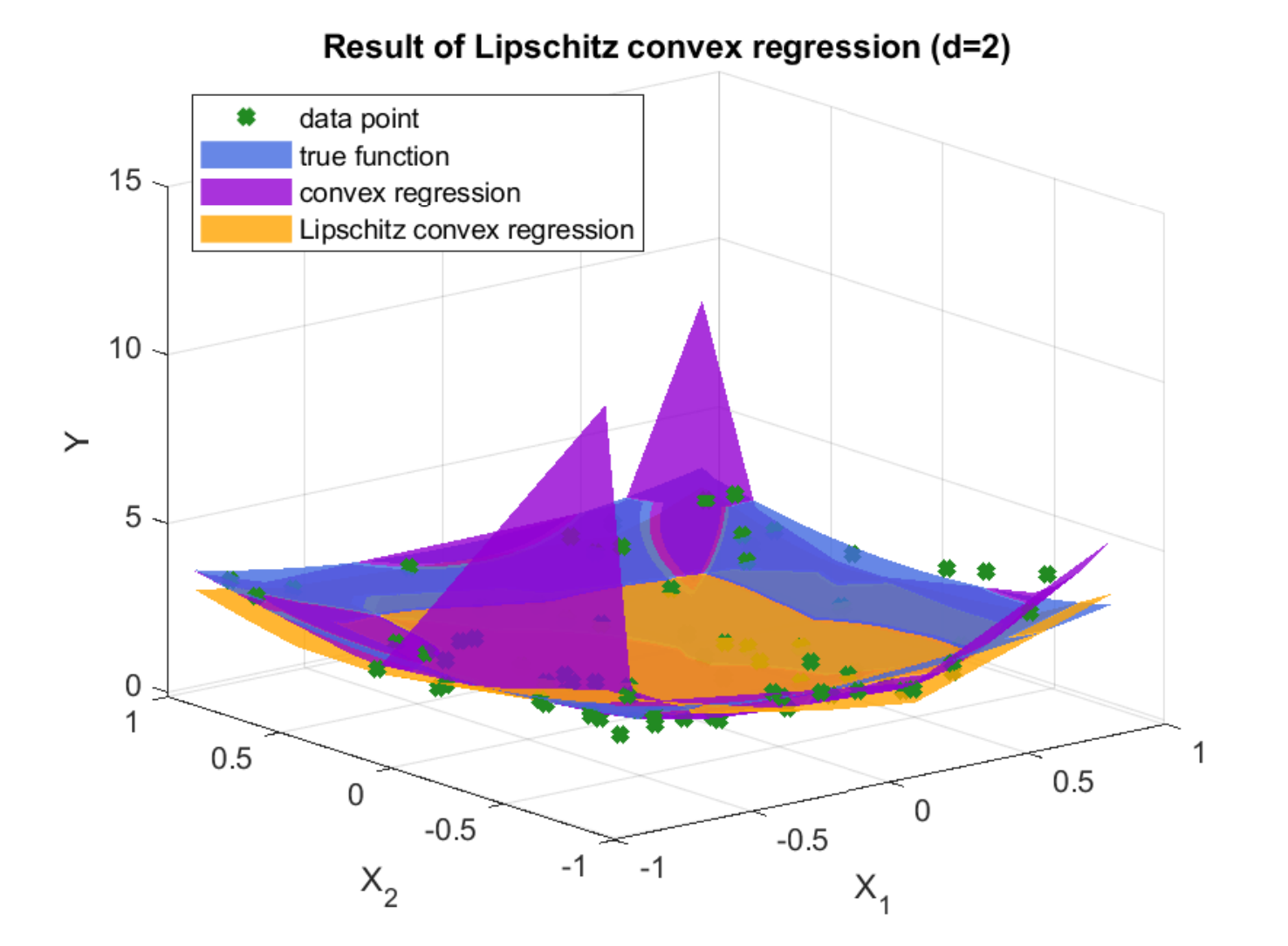}}
	\caption{Result of Lipschitz convex regression with the data-driven Lipschitz estimation method.}\label{fig:data_Lip}
\end{figure}

\section{Real applications}\label{sec:realapplication}
In this section, we apply our mechanism for estimating the multivariate shape-constrained convex functions in some real applications, namely, pricing of European call option, pricing of basket option, prediction of average weekly wages and estimation of production functions.

\subsection{Option pricing of European call option}\label{subsec:calloption}
Consider a European call option whose payoff at maturity $T$ is $ (S_T-K)_+$, where $S_T$ is a random variable that stands for the stock price at $T$, and $K$ is the predetermined strike price. We are interested in the option price at time $t$, which is defined as
\begin{align*}
V(S) := \mathbb{E}[e^{-r(T-t)}(S_T-K)_+\mid S_t=S], \quad S>0,
\end{align*} 
where $r$ is the risk-free interest rate. Under Black-Scholes model, the random variable $S_T$ satisfies
\begin{align*}
\log S_T \sim \mathcal{N}\Big(\log S_t+(r-\frac{1}{2}\sigma^2)(T-t),\sigma^2(T-t)
\Big),
\end{align*}
where $\sigma$ is the volatility. It is well-known that $V(\cdot)$ is a convex function with $0\leq V'(S)\leq 1$ for $S>0$. Therefore, we can use the shape-constrained convex regression model with Property (S2) to approximate the function $V(\cdot)$. 

There are two reasons why we consider this application to demonstrate the numerical performance of our mechanism. One is that $V(\cdot)$ admits a closed-form solution as
\begin{align*}
V(S)=S\Phi(d_1)-Ke^{-r(T-t)}\Phi(d_2),\quad
d_{1,2} = \frac{\log \frac{S}{K}+(r\pm \frac{1}{2}\sigma^2)(T-t)}{\sigma \sqrt{T-t}},
\end{align*}
where $\Phi(\cdot)$ is the cumulative distribution function of the standard normal distribution. The second reason is that the approximation of function $V(\cdot)$ is used in pricing American-type options by approximate dynamic programming, e.g. \cite{longstaff2001valuing}.

In our experiment, we take $t=0.1$, $T=0.4$, $K=10$, $r=0$, $\sigma=0.2$. We sample $200$ data points, denoted as $\{(S_i,V_i)\}_{i=1}^{200}$. For each $S_i$, $\log S_i$ is sampled following the distribution $\mathcal{N}(\log K+(r-\sigma^2/2)t,\sigma^2t)$, and the corresponding $V_i$ is sampled such that $\log V_i$ follows the distribution $\mathcal{N}(\log S_i+(r-\sigma^2/2)(T-t),\sigma^2(T-t))$. For comparison, we apply several regression models to approximate the conditional expectation function $V$: linear regression, least squares linear regression on a set of basis functions (e.g. weighted Laguerre basis in \cite{longstaff2001valuing}), unconstrained convex regression and convex regression with box constraint ($L=0$, $U=1$). 

The comparison among regression models is shown in Figure \ref{toy_Lip}. We can see that the performance of shape-constrained convex regression is the best. The poor performance of linear regression, Laguerre regression and unconstrained convex regression appears near the boundary in three aspects. The first is that the results from linear regression and Laguerre regression take negative values when $S$ is small, which contradicts the fact that $V$ is always non-negative. The second is that the Laguerre regression can not obtain the required convex property. The last is that when $S$ is large, the gradients of the results obtained by Laguerre regression and unconstrained convex regression are too large. To deal with this over-fitting problem, we add the box constraint to the convex regression, which comes from prior knowledge. We can see that the result of shape-constrained convex regression performs better near the boundary, which demonstrate the advantage of the additional shape constraint.

\begin{figure}[H]
	\centering
	\includegraphics[width=0.8\linewidth]{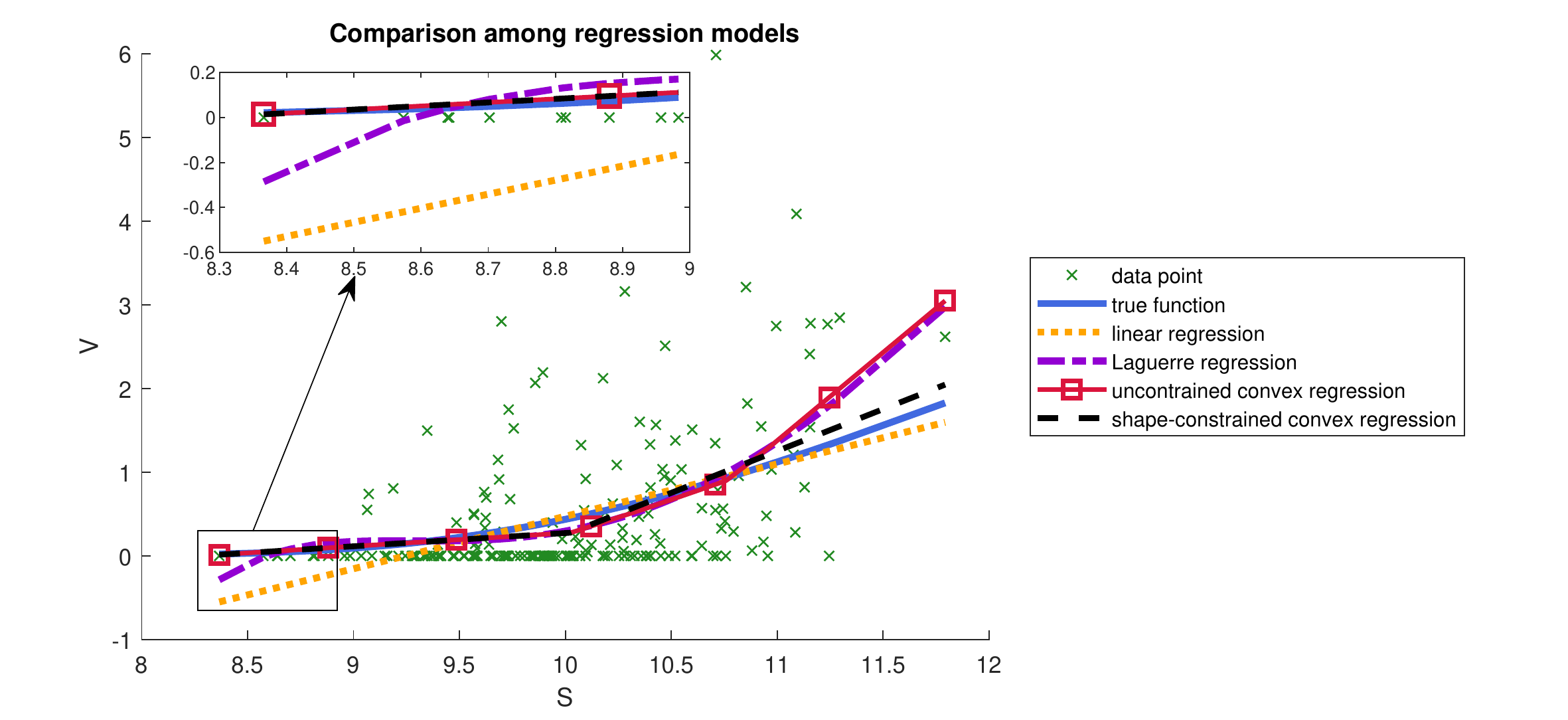}
	\caption{Results of the estimation of the option pricing of European call option.}\label{toy_Lip}
\end{figure}

\subsection{Option pricing of basket option}\label{subsec:basketoption}
To test multivariate convex regression problem, we consider pricing the basket option on weighted average of $M$ underlying assets. 

\paragraph{Basket option of two European call options ($M=2$).}
We first consider a basket option of two European call options, where
\begin{align*}
V(x,y)=\mathbb{E}[e^{-r(T-t)}(w_1 S_T^1+w_2S_T^2-K)_{+}\mid S_t^1=x,S_t^2=y],\quad x,y>0,
\end{align*}
where $w=(w_1,w_2)^T$ is a given weight vector such that $w\geq 0$, $w_1+w_2=1$. The random variables $S_T^1$ and $S_T^2$ satisfy
\begin{align*}
\left(\begin{aligned}
&\log S_T^1\\
&\log S_T^2
\end{aligned}
\right)\sim \mathcal{N}
\left(\left(\begin{aligned}
&\log S_t^1+(r-\sigma_1^2/2)(T-t)\\
&\log S_t^2+(r-\sigma_2^2/2)(T-t)
\end{aligned}
\right),(T-t)\begin{pmatrix}
\sigma_1^2 & \rho\sigma_1\sigma_2\\[0.4em]
\rho\sigma_1\sigma_2 &\sigma_2^2
\end{pmatrix}
\right),
\end{align*}
where $\sigma_1$, $\sigma_2$ are volatilities. One can show that $V(\cdot,\cdot)$ is a convex function with $0\leq \nabla V(x,y)\leq w$, and the proof can be found in the Appendix. We can apply the multivariate shape-constrained convex regression model with Property (S2) ($L=0$, $U=w$) to estimate the function $V$. 

The convex function $V(\cdot,\cdot)$ does not admit a closed-form solution. However $V$ is also the solution of Black-Scholes PDE, which can be solved by the finite difference method. The details of the corresponding convection-diffusion equation and the finite difference method for solving it could be found in the Appendix. We use the solution obtained by the finite difference method as the benchmark. 

In the experiment, we take $r=0$, $\rho=0.1$, $\sigma_1=0.2$, $\sigma_2=0.3$, $K=10$, $t=0$, $T=0.5$, $w_1=w_2=0.5$. We sample $200$ data points, denoted as $\{(S_i,V_i)\}_{i=1}^{200}$, where $S_i$ follows the uniform distribution on the open interval $(0,5K)\times (0,5K)$ and $V_i$ follows the the distribution
\begin{align*}
\mathcal{N}
\left(\log S_i+(T-t)\left(\begin{aligned}
&r-\sigma_1^2/2\\
&r-\sigma_2^2/2
\end{aligned}
\right),(T-t)\begin{pmatrix}
\sigma_1^2 & \rho\sigma_1\sigma_2\\[0.4em]
\rho\sigma_1\sigma_2 &\sigma_2^2
\end{pmatrix}
\right).
\end{align*}
The numerical result is shown in Figure \ref{fig:basket}. For better illustration, we also plot the absolute error and relative error of the results of the unconstrained convex regression and shape-constrained convex regression. As we can see, the shape-constrained convex regression performs much better than unconstrained convex regression, especially near the boundary.

\begin{figure}[H]
	\centering
	\subfigure[]{\label{fig:basket1}
		\includegraphics[width=0.8\linewidth]{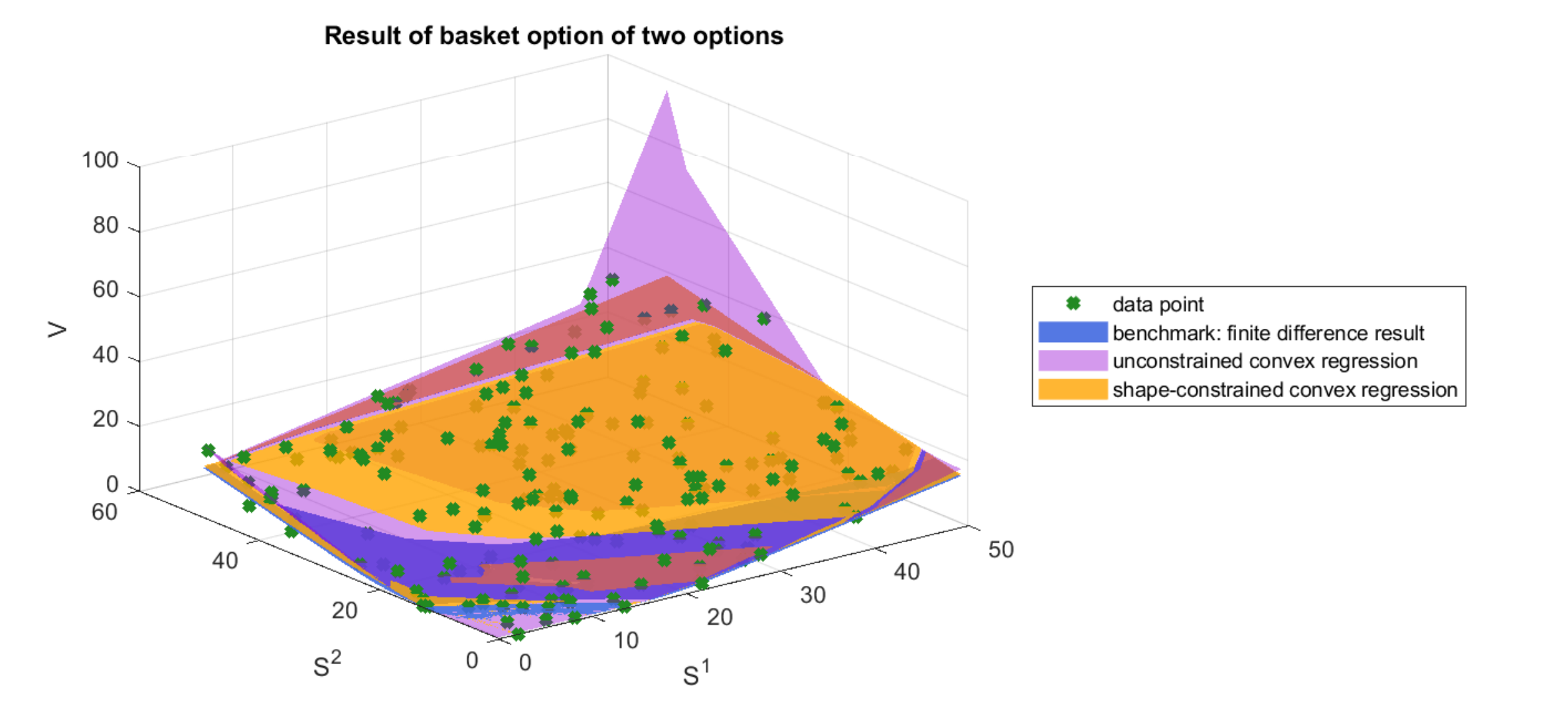}}\\
	\subfigure[]{\label{fig:basket2}
		\includegraphics[width=0.485\linewidth]{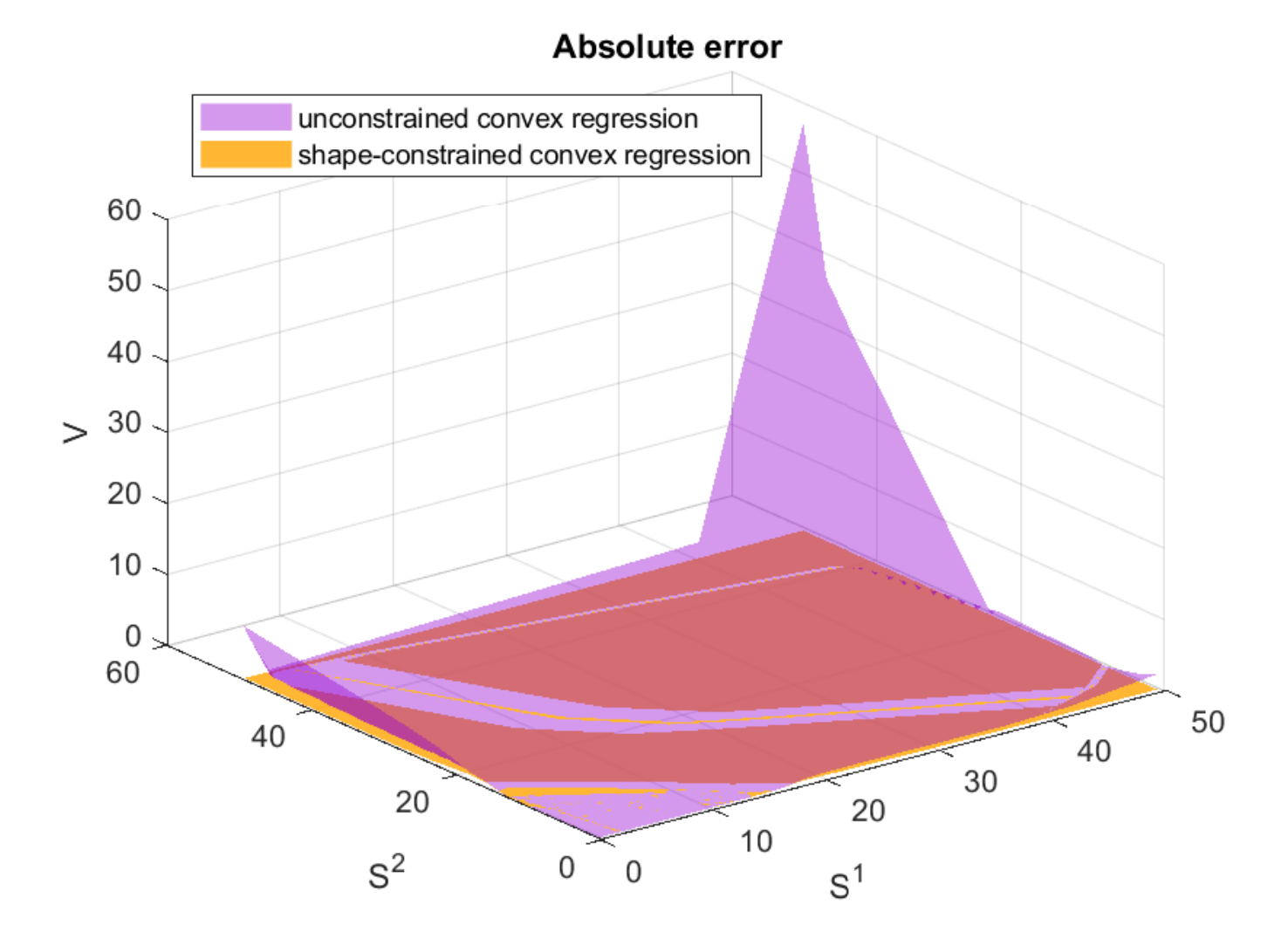}}
	\subfigure[]{\label{fig:basket3}
	    \includegraphics[width=0.485\linewidth]{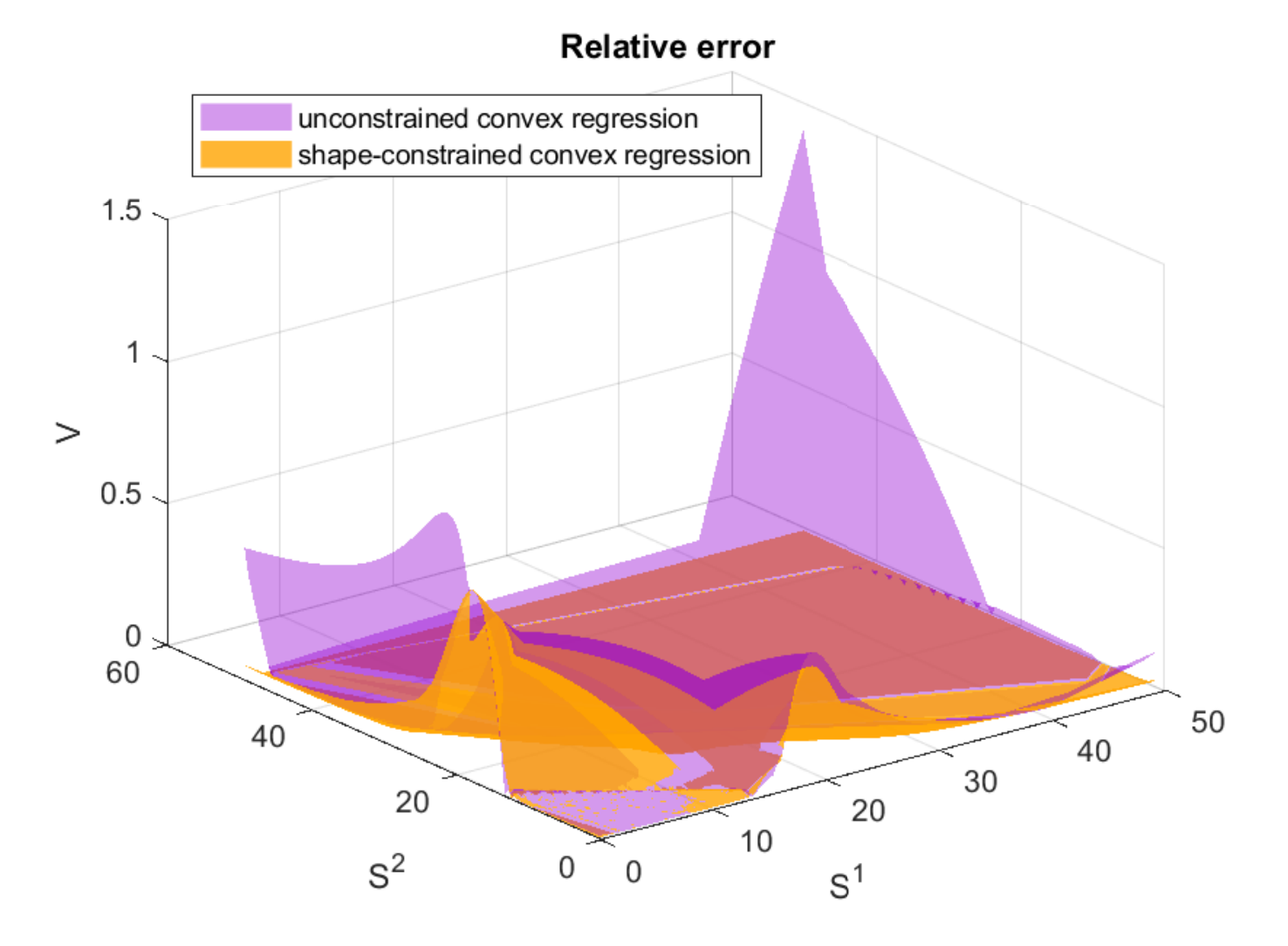}}
	\caption{Result of the estimation of the option pricing of basket option ($M=2$).}\label{fig:basket}
\end{figure}

\paragraph{Basket option of more underlying assets ($M>2$).}
The basket option in practice always contains many underlying assets, possibly greater than two. The finite difference method is very time-consuming when solving the 
$3$-dimensional convection-diffusion equation, and even impossible to be applied to the higher dimensional cases due to the curse of dimensionality. For $M>2$, researchers tend to apply the Monte Carlo simulation to estimate the convex function associated with the basket option. Therefore, we treat the solution obtained by the Monte Carlo simulation as the benchmark. 

To demonstrate the performance of the shape-constrained convex regression, we design the experiments for estimating the basket option for $M=5$ and $M=10$. That is, we consider a basket option of $M$ European call options, which is defined as
\begin{align*}
V(x_1,\cdots,x_M)=\mathbb{E}[e^{-r(T-t)}(w_1 S_T^1+\cdots+w_MS_T^M-K)_{+}\mid S_t^1=x_1,\cdots,S_t^M=x_M],\quad x_1,\cdots,x_M>0,
\end{align*}
where $w=(w_1,\cdots,w_M)^T$ is a given weight vector such that $w\geq 0$, $w_1+\cdots+w_M=1$. The random variables $S_T^1,\cdots,S_T^M$ satisfy 
\begin{align*}
\begin{pmatrix}
\log S_T^1 \\[0.4em]
\vdots \\[0.4em]
\log S_T^M
\end{pmatrix}\sim \mathcal{N}
\left(\begin{pmatrix}
\log S_t^1+(r-\sigma_1^2/2)(T-t)\\[0.4em]
\vdots \\[0.4em]
\log S_t^M+(r-\sigma_M^2/2)(T-t)
\end{pmatrix},(T-t)\begin{pmatrix}
\sigma_1^2 & \cdots & \rho\sigma_1\sigma_M\\[0.4em]
\vdots & \ddots & \vdots\\[0.4em]
\rho\sigma_1\sigma_M& \cdots  &\sigma_M^2
\end{pmatrix}
\right),
\end{align*}
where $\sigma_1,\cdots,\sigma_M$ are volatilities. Then $V$ is a convex function with $0\leq \nabla V\leq w$. We apply the multivariate shape-constrained convex regression model with Property (S2) ($L=0$, $U=w$) to estimate the function $V$. 

In the experiment, we set $r=0$, $\rho=0.1$, $K=10$, $t=0$, $T=0.5$, $w_i=1/M$, $\sigma_i=0.2+0.025(i-1)$, $i=1,\cdots,M$. We sample $N$ data points as the case for $M=2$. To illustrate the performance of our procedure, we uniformly generate $1000$ test points in the range $(0,5K)^M$. At each test point, we use the Monte Carlo simulation with $10^5$ samples to compute the ``true" function value. We summarize the results of $M=5$ and $M=10$ in Table \ref{table_M5} and Table \ref{table_M10}, respectively. In the tables, ``UC" represents the unconstrained convex regression, ``SC" represents the shape-constrained convex regression, and ``MSE" represents the mean squared error. As one can see, the shape-constrained convex regression takes longer time to be solved than the unconstrained convex regression, but get a much better estimated result.

\begin{minipage}[b]{0.45\linewidth}\centering
{\small
	\begin{ThreePartTable}
		\begin{TableNotes}
			\item[\bfseries {\footnotesize Note:} ]  {\footnotesize MSE represents the mean squared error.}
		\end{TableNotes}
		\renewcommand{\multirowsetup}{\centering}
		\begin{longtable}{cccc}
			\caption{Estimation of basket option with $M=5$.} \label{table_M5}\\[-6pt]
			\hline  
			\specialrule{0.01em}{1pt}{1pt}
			Model & Num. of data & MSE &  Time \\
			\specialrule{0.01em}{0pt}{2pt}
			\endfirsthead
			
			\multicolumn{4}{c}{{ \tablename\ \thetable{} -- continued from previous page}} \\
			
			\specialrule{0.01em}{1pt}{1pt}
			Model & Num. of data & MSE &  Time \\
			\specialrule{0.01em}{0pt}{2pt}
			\endhead
			
			\hline
			\multicolumn{4}{r}{{Continued on next page}} \\
			\hline \hline
			\endfoot
			
			\specialrule{0.01em}{0pt}{1pt} 
			\hline
			\insertTableNotes
			\endlastfoot	
			
			\multirow{3}*{\tabincell{c}{\tt UC}} 
			& 200 & 1.35e+2 & 00:00:05 \\
			& 400 & 2.06e+1 & 00:00:38 \\
			& 600 & 7.56e+1 & 00:01:18 \\
			\specialrule{0.01em}{1pt}{1pt}
			\multirow{3}*{\tabincell{c}{\tt SC}} 
			& 200 & 4.34e-1 & 00:00:24 \\
			& 400 & 3.97e-1 & 00:01:36 \\
			& 600 & 5.83e-1 & 00:02:34 \\
			\hline
			\specialrule{0.01em}{0.5pt}{1pt}
		\end{longtable}
	\end{ThreePartTable}
}	
\end{minipage}
\hspace{0.5cm}
\begin{minipage}[b]{0.45\linewidth}\centering
{\small
	\begin{ThreePartTable}
		\begin{TableNotes}
			\item[\bfseries {\footnotesize Note:} ]  {\footnotesize MSE represents the mean squared error.}
		\end{TableNotes}
		\renewcommand{\multirowsetup}{\centering}
		\begin{longtable}{cccc}
			\caption{Estimation of basket option with $M=10$.} \label{table_M10}\\[-6pt]
			\hline  
			\specialrule{0.01em}{1pt}{1pt}
			Model & Num. of data & MSE &  Time \\
			\specialrule{0.01em}{0pt}{2pt}
			\endfirsthead
			
			\multicolumn{4}{c}{{ \tablename\ \thetable{} -- continued from previous page}} \\
			
			\specialrule{0.01em}{1pt}{1pt}
			Model & Num. of data & MSE &  Time \\
			\specialrule{0.01em}{0pt}{2pt}
			\endhead
			
			\hline
			\multicolumn{4}{r}{{Continued on next page}} \\
			\hline \hline
			\endfoot
			
			\specialrule{0.01em}{0pt}{1pt} 
			\hline
			\insertTableNotes
			\endlastfoot	
			
			\multirow{3}*{\tabincell{c}{\tt UC}} 
			& 200 & 4.60e+2 & 00:00:08 \\
			& 400 & 9.36e+1 & 00:00:10 \\
			& 600 & 4.29e+1 & 00:00:28 \\
			\specialrule{0.01em}{1pt}{1pt}
			\multirow{3}*{\tabincell{c}{\tt SC}} 
			& 200 & 2.91e+0 & 00:00:29 \\
			& 400 & 1.45e+0 & 00:01:04 \\
			& 600 & 1.49e+0 & 00:03:52 \\
			\hline
			\specialrule{0.01em}{0.5pt}{1pt}
		\end{longtable}
	\end{ThreePartTable}
}	
\end{minipage}
	
\subsection{Prediction of average weekly wages}\label{subsec:weeklywages}
We consider the problem of estimating the average weekly wages based on years of education and experience as given in \cite[Chapter 10]{ramsey2012statistical}. This dataset is from 1988 March U.S. Current Population Survey, which can be downloaded as ex1029 in the R package Sleuth2. The set contains weekly wages in 1987 for a sample of 25632 males between the age of 18 and 70 who worked full-time, with their years of education and years of experience. After averaging over a grid with cell size of $1$ year by $1$ year and ignoring the outliers, we finally come to a dataset with $857$ samples. 

A reasonable assumption for this application is that the wages are concave in years of experience and a transformation of years of education, i.e., $1.2^{\mbox{years of education}}$, according to \cite{hannah2013multivariate}.
The estimated result is shown in Figure \ref{fig:ex1029}. The shape-constrained convex regression problem is solved within $1$ minute.

\begin{figure}[H]
	\subfigure[Estimated function values at each $X_i$.]{\label{fig:ex1}
		\includegraphics[width=0.485\linewidth]{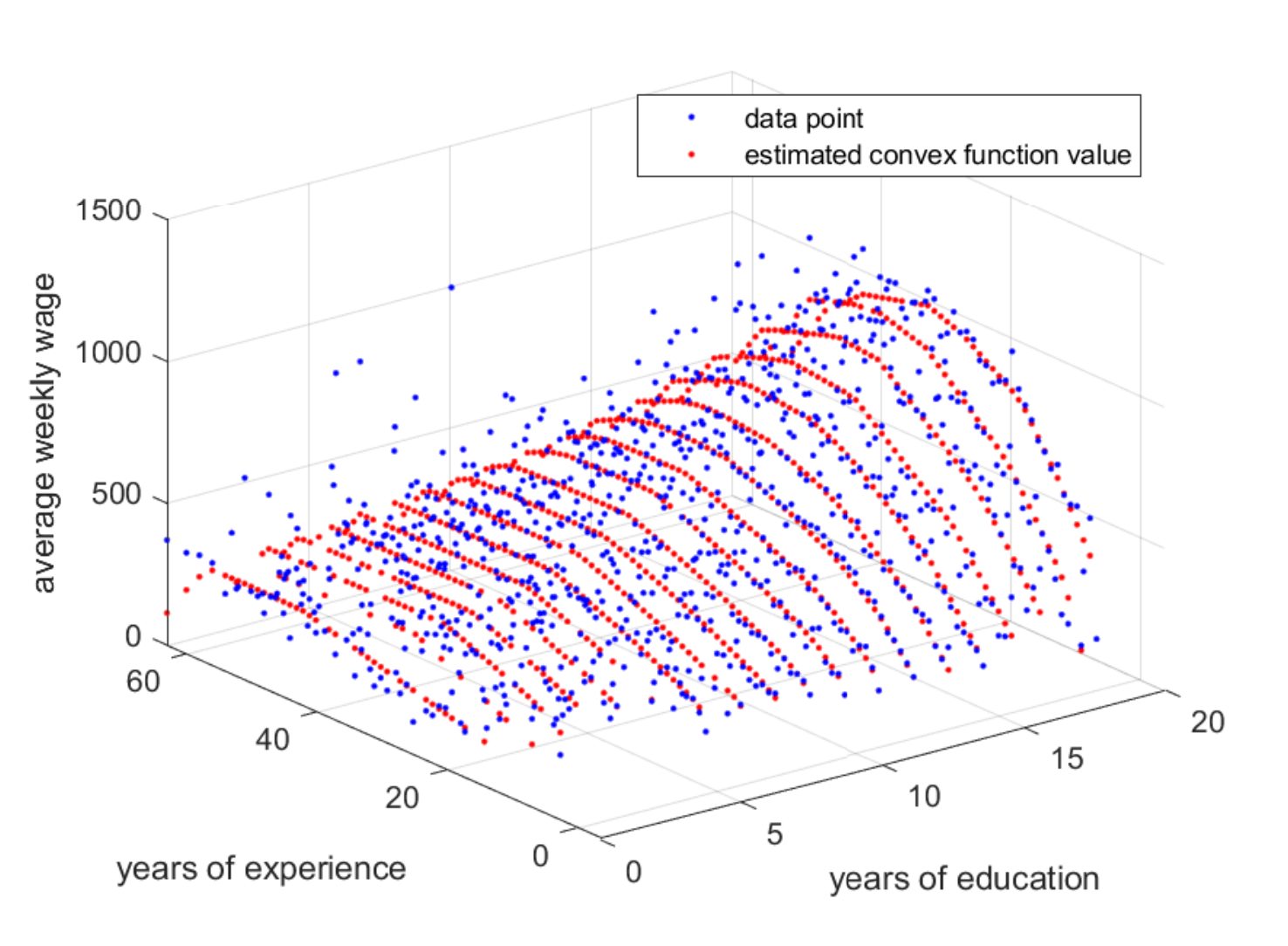}}
	\subfigure[Visualization of the function.]{\label{fig:ex2}
		\includegraphics[width=0.485\linewidth]{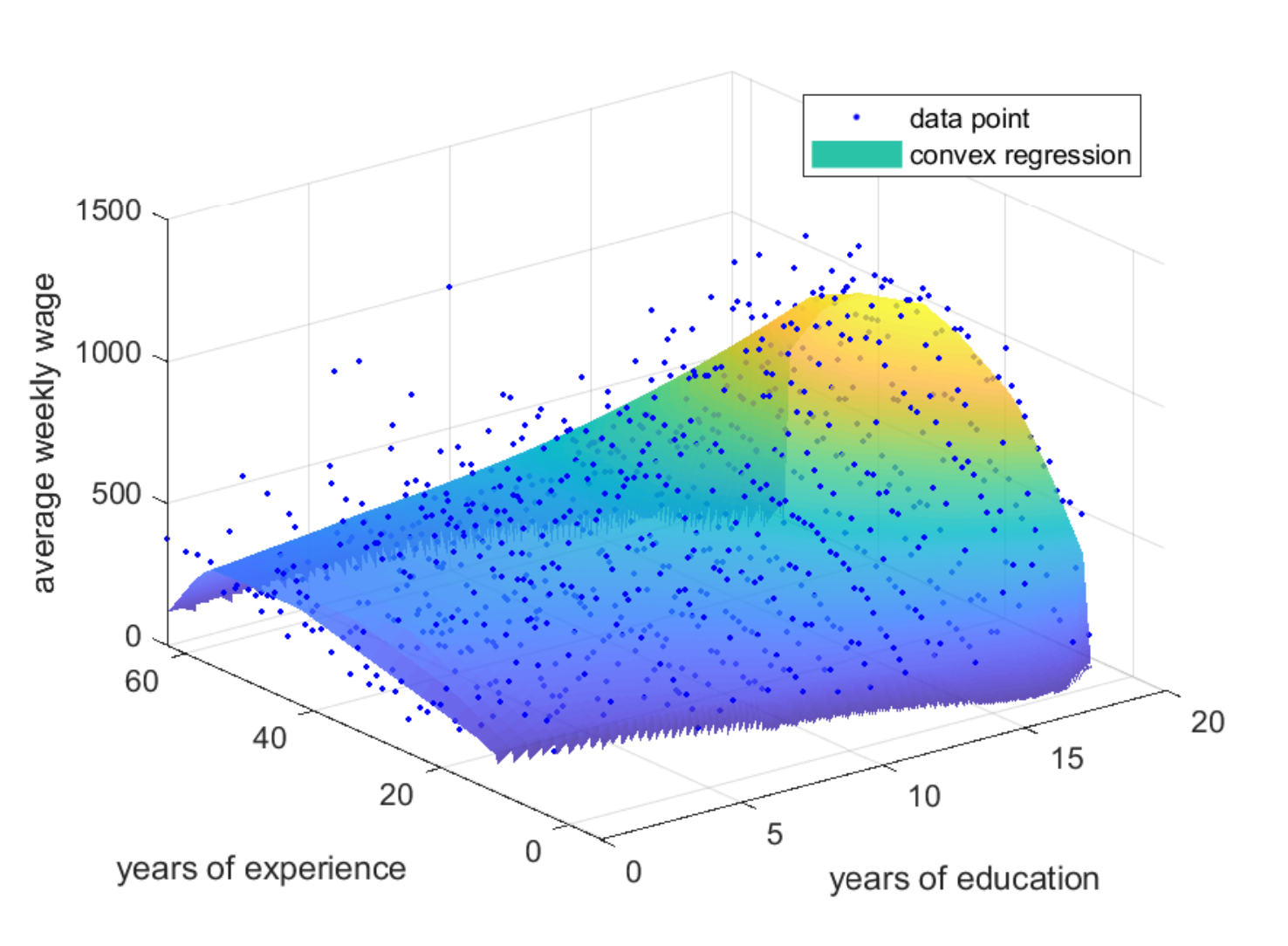}}
	\caption{Results of the estimation of average weekly wages.}\label{fig:ex1029}
\end{figure}

\subsection{Estimation of production functions}\label{subsec:productionfun}
In economics, a production function gives the technological relation between quantities of inputs and quantities of output of goods. Production functions are known to be concave and non-decreasing \cite{hanoch1972testing,varian1984nonparametric,yagi2018shape}. We apply our framework to estimate the production function for the plastic industry (CIIU3 industry code: 2520) in the year 2011. The dataset can be downloaded from the website of Chile's National Institute of Statistics ({\tt https://www.ine.cl/estadisticas/economicas/manufactura}). As in the setting in \cite{yagi2018shape}, we use labor and capital as the input variables, and value added as the output variable. In the dataset, labor is measured as the total man-hours per year, capital and value added are measured in millions of Chilean peso. After removing some outliers, the dataset contains 250 samples. The numerical results can be found in Figure \ref{fig:prod1}. The shape-constrained convex regression problem is solved within $3$ seconds.

\begin{figure}[H]
	\subfigure[Correlogram of the function.]{\label{fig:prod4_3}
		\includegraphics[width=0.5\linewidth]{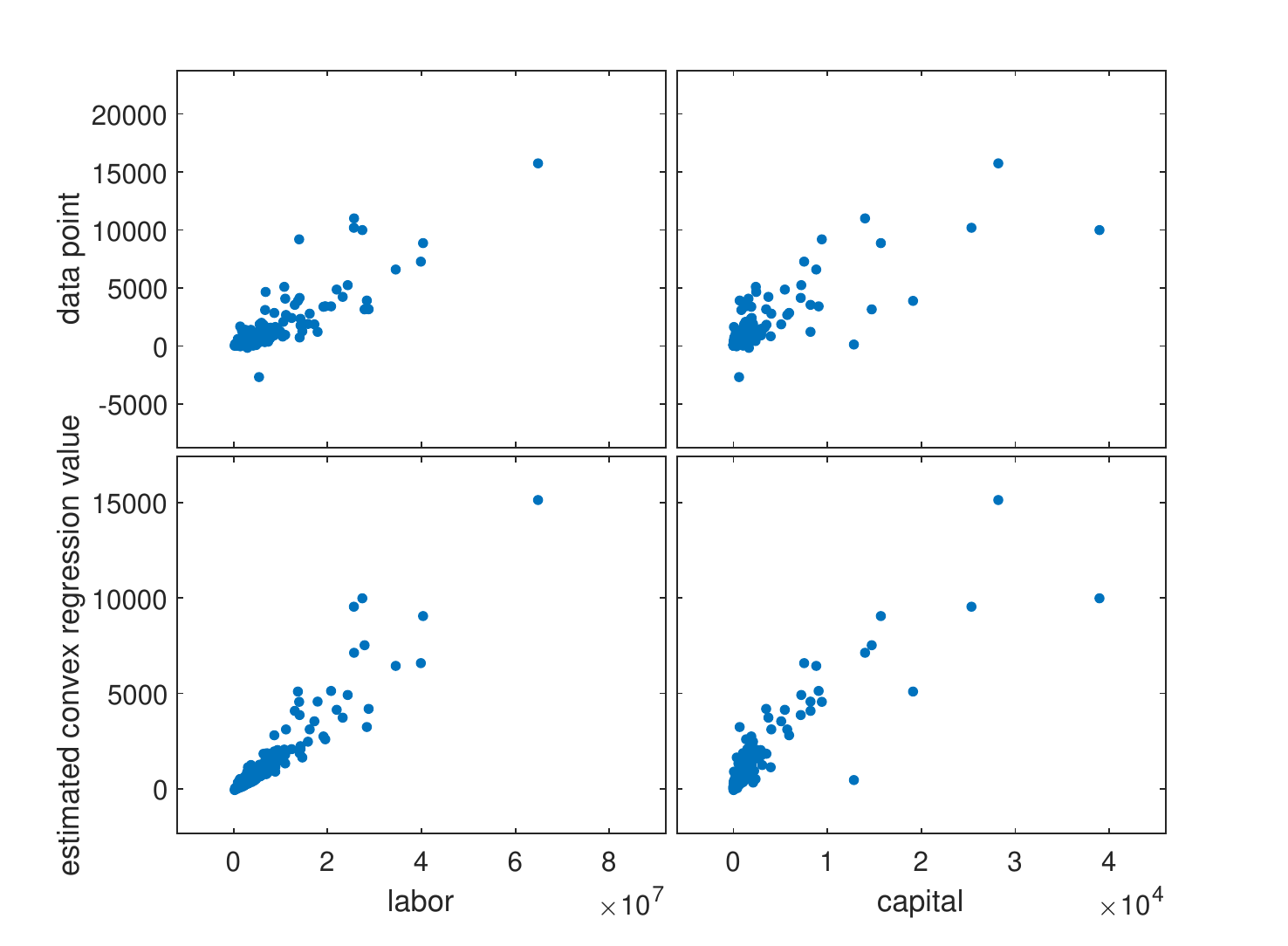}}
	\subfigure[Visualization of the function.]{\label{fig:prod4_4}
		\includegraphics[width=0.5\linewidth]{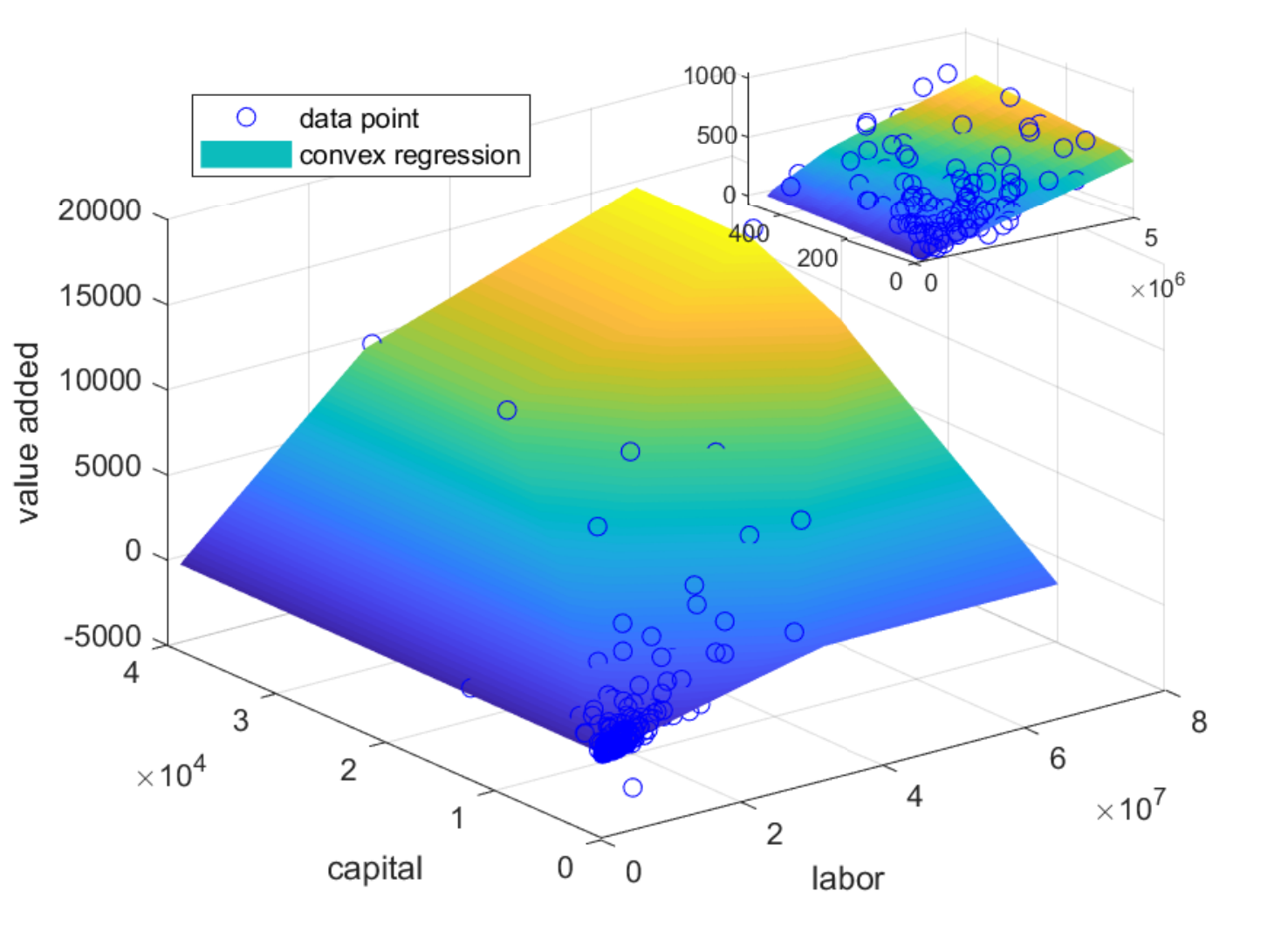}}	
	\caption{Result of estimation of production function of plastic in Chile.}\label{fig:prod1}
\end{figure}

Another example is to explain the labour demand of 569 Belgian firms for the year 1996. The dataset can be obtained from \cite{verbeek2008guide} ({\tt https://www.wiley.com/legacy/wileychi/verbeek2ed/datasets.html}). The dataset includes the total number of employees (labour), their average wage (wage), the amount of capital (capital) and a measure of output (value added). The labour is measured as the number of workers, the wage is measured in units of 1000 euro, and the capital and value added is measured in units of a million euro. After removing the outliers, the dataset contains 562 samples. The result can be found in Figure \ref{fig:prod2}. The problem is solved in $22$ seconds.

\begin{figure}[H]
	\subfigure[Correlogram of the function.]{\label{fig:prod9_4}
		\includegraphics[width=0.5\linewidth]{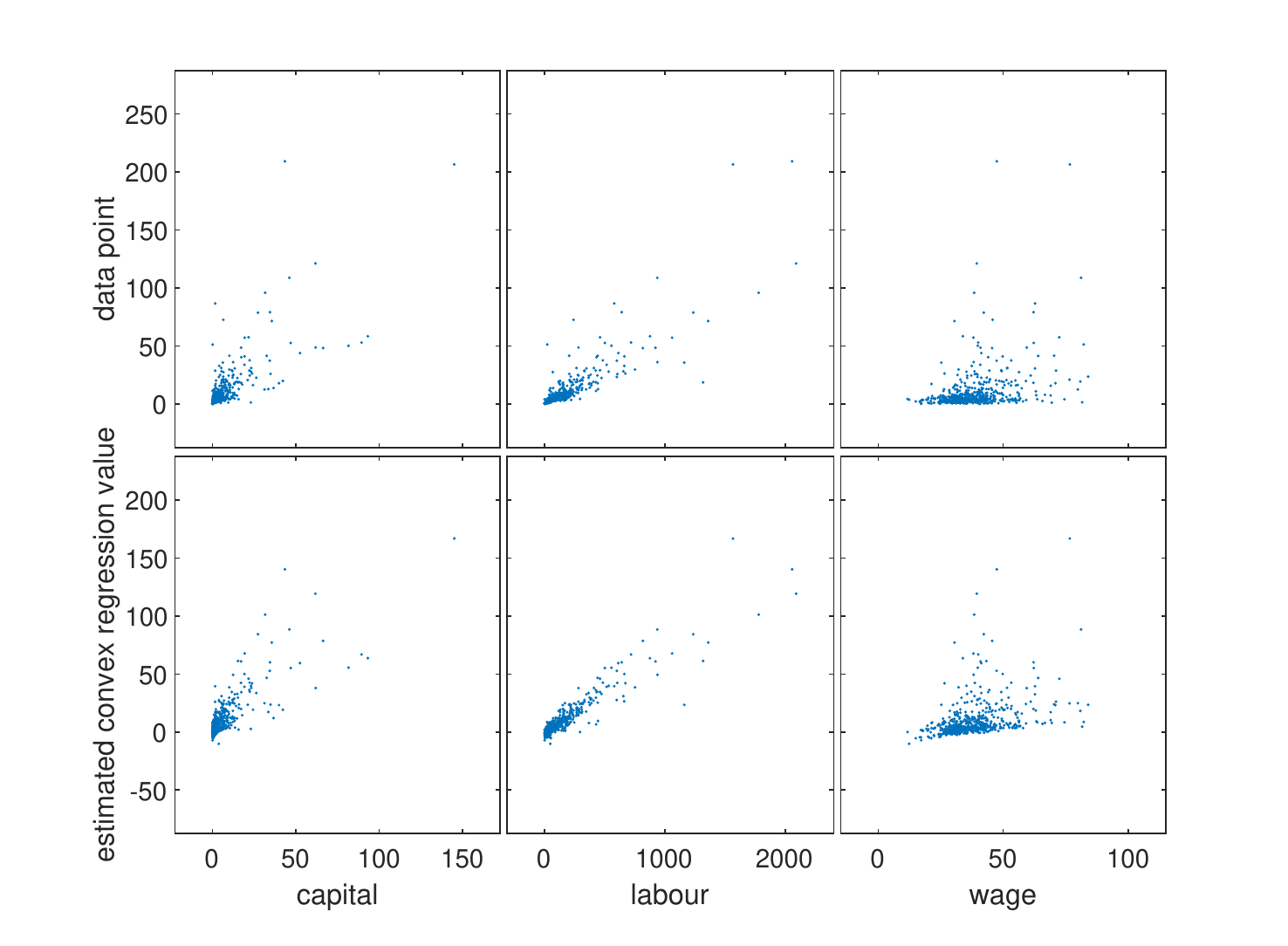}}
	\subfigure[Visualization of the function.]{\label{fig:prod9_5}
		\includegraphics[width=0.5\linewidth]{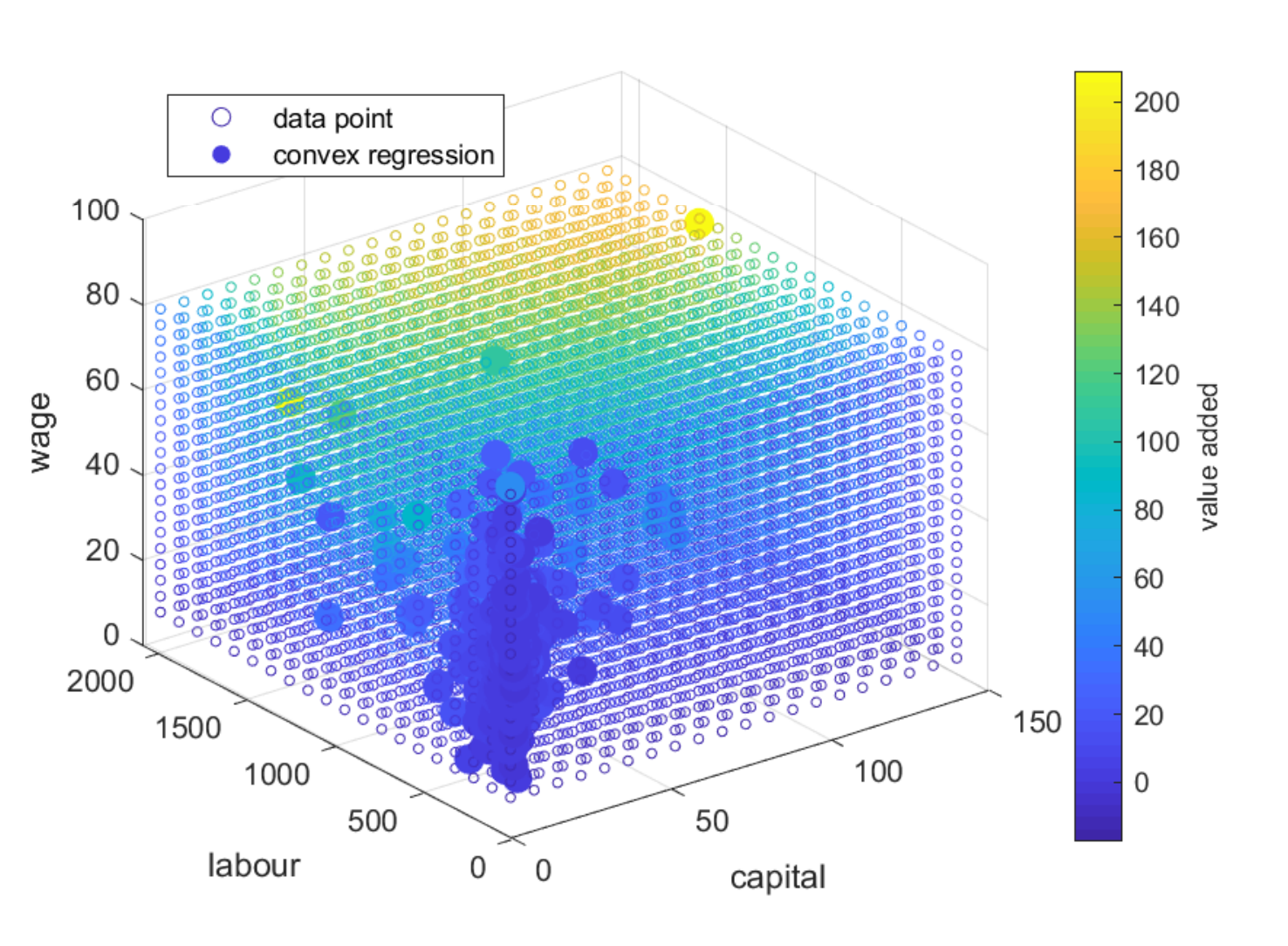}}	
	\caption{Result of estimation of production function of Belgian firms.}\label{fig:prod2}
\end{figure}

\section{Conclusion and future work}\label{sec:futurework}
In this paper, we provide a comprehensive mechanism for computing a least squares estimator for the multivariate shape-constrained convex regression function.
In addition, we propose two efficient algorithms, the symmetric Gauss-Seidel based alternating direction method of multipliers, and the proximal augmented Lagrangian method, to solve the large-scale constrained QP in the mechanism. We conduct the extensive numerical experiments to demonstrate the efficiency and robustness of our proposed algorithms. In future, we may extend the idea to estimate a continuously differentiable function with a Lipschitz derivative, which is known to be the difference of a convex function and a convex quadratic function on a compact convex set \cite{zlobec2008fundamental}.

\section*{Acknowledgements}
The authors would like to thank Professor Necdet S. Aybat for helpful clarifications on his work in \cite{aybat2014parallel}.

\appendix
\section*{Appendices}
\section{Derivation of the proximal mapping and generalized Jacobian associated with $\mathcal{D}=\{x\in \mathbb{R}^d\mid \|x\|_{1}\leq L\}$}
For $x\in \mathbb{R}^d$, let $P_x={\rm diag}({\rm sign}(x))\in \mathbb{R}^{d\times d}$, then
\begin{align*}
\Pi_\mathcal{D}(x) &= \arg\min_{y\in\mathbb{R}^d} \Big\{\frac{1}{2}\|y-x\|^2\mid \|y\|_1\leq L\Big\}\\
&= LP_x \Big(\arg\min_{y\in\mathbb{R}^d} \Big\{\frac{1}{2}\|y-P_x x/L\|^2\mid e_d^T y \leq 1,y\geq 0\Big\}\Big)\\
&=\left\{\begin{aligned}
&x && \mbox{if}\ \|x\|_1\leq L, \\
&L P_x \Pi_{\Delta_d}(P_x x/L)&& \mbox{otherwise,}
\end{aligned}\right.
\end{align*}
where the simplex $\Delta_d=\{x\in \mathbb{R}^d\mid e_d^T x=1,x\geq 0\}$. To derive the generalized Jacobian of $\Pi_\mathcal{D}(\cdot)$, we need the generalized Jacobian of $\Pi_{\Delta_d}(\cdot)$. Following the idea in \cite{han1997newton,li2017efficient}, we can explicitly compute an element of the generalized Jacobian of $\Pi_{\Delta_d}(\cdot)$ at $P_x x/L$. Let $K$ be the set of index $i$ such that $(\Pi_{\Delta}(P_x x/L))_i=0$. Then
\begin{align*}
\widetilde{H}=I_d-\begin{bmatrix}
I_K^T & e_d
\end{bmatrix}\Big(
\begin{bmatrix}
I_K \\[0.4em]
e_d^T
\end{bmatrix}
\begin{bmatrix}
I_K^T & e_d
\end{bmatrix}
\Big)^{\dagger}\begin{bmatrix}
I_K \\[0.4em]
e_d^T
\end{bmatrix}
\end{align*}
is an element in $\partial \Pi_{\Delta_d}(P_x x/L)$, where $I_K$ means the matrix consisting of the rows of the identity matrix $I_d$, indexed by $K$. After some algebraic computation, we can see
\begin{align*}
\widetilde{H}&=I_d-\begin{bmatrix}
I_K^T & e_d
\end{bmatrix}
\begin{bmatrix}
I_{|K|}+\frac{1}{n-|K|}e_{|K|} e_{|K|}^T & -\frac{1}{n-|K|}e_{|K|} \\[0.4em] 
-\frac{1}{n-|K|}e_{|K|}^T & \frac{1}{n-|K|}
\end{bmatrix}
\begin{bmatrix}
I_K \\[0.4em] 
e_d^T
\end{bmatrix}\\
&={\rm Diag}(r)-\frac{1}{{\rm nnz}(r)}rr^T,
\end{align*}
where $r\in \mathbb{R}^d$ is defined as $r_i=1$ if $(\Pi_{\Delta}(P_x x/L))_i\neq 0 $ and $r_i=0$ otherwise. Therefore,
\begin{align*}
H\in \partial \Pi_\mathcal{D}(x),\quad \mbox{where } H=\left\{\begin{aligned}
&I_d && \mbox{if}\ \|x\|_1\leq L, \\
&P_x \widetilde{H}  P_x && \mbox{otherwise}.
\end{aligned}\right.
\end{align*}

\section{Property of basket option of two European call options}
The function $V(x,y)$ is differentiable since it is the solution of the Black-Scholes PDE. By the definition of $V$, we can see that $V$ is non-decreasing in $x$ and $y$, which means that $\nabla V(x,y)\geq 0$. According to the distribution of $S_T^1$ and $S_T^2$, we have that 
\begin{align*}
V(x,y) = e^{-r(T-t)} \mathbb{E}_z f(x,y,z),
\end{align*}
where 
\begin{align*}
f(x,y,z)&=(xw_1 e^{(r-\sigma_1^2/2)(T-t)+\sqrt{T-t}z_1}+yw_2 e^{(r-\sigma_2^2/2)(T-t)+\sqrt{T-t}z_2}-K)_+,\\
\begin{pmatrix}
z_1\\[0.4em]
z_2
\end{pmatrix}
&\sim \mathcal{N}(0,\begin{pmatrix}
\sigma_1^2 & \rho\sigma_1\sigma_2\\[0.4em]
\rho\sigma_1\sigma_2 &\sigma_2^2
\end{pmatrix}).
\end{align*}
For any $x_1,x_2,y\in\mathbb{R}$, we can see that 
\begin{align*}
|V(x_1,y)-V(x_2,y)|&=e^{-r(T-t)} \Big| \mathbb{E}_z [f(x_1,y,z)-f(x_1,y,z)]\Big|\\
&\leq e^{-r(T-t)} \mathbb{E}_z |f(x_1,y,z)-f(x_1,y,z)|\\
&\leq e^{-r(T-t)} \mathbb{E}_z  [w_1e^{(r-\sigma_1^2/2)(T-t)+\sqrt{T-t}z_1} |x_1-x_2|]\\
&=w_1|x_1-x_2|e^{-\sigma_1^2/2(T-t)}\mathbb{E}_z [e^{\sqrt{T-t}z_1}]\\
&=w_1|x_1-x_2|.
\end{align*}
Similarly, we can prove that for any $x,y_1,y_2\in\mathbb{R}$,
\begin{align*}
|V(x,y_1)-V(x,y_2)|\leq w_2|y_1-y_2|.
\end{align*}
Therefore, we have that fact that $0\leq \nabla V(x,y)\leq w$ for any $x,y$.

\section{Finite difference method for estimating the basket option of two European call options}
It is well-known that the function $V(x,y)=U(0,x,y)$, where $U$ satisfies the Black-Scholes PDE
\begin{align*}
\left\{\begin{aligned}
&\frac{\partial U}{\partial t}+rx\frac{\partial U}{\partial x}+ry\frac{\partial U}{\partial y}+\frac{1}{2}\sigma_1^2x^2\frac{\partial^2 U}{\partial^2 x^2}+\rho \sigma_1\sigma_2xy\frac{\partial^2 U}{\partial xy}+\frac{1}{2}\sigma_2^2y^2\frac{\partial^2 U}{\partial^2 y^2}-rU=0,\\
&U(T,x,y)=(w_1x+w_2y-K)^+.
\end{aligned}
\right.
\end{align*}
Let $\tau = T-t$, $u(\tau,x,y)=U(t,x,y)$, then $u$ satisfies
\begin{align*}
\left\{\begin{aligned}
&\frac{\partial u}{\partial \tau}-rx\frac{\partial u}{\partial x}-ry\frac{\partial u}{\partial y}-\frac{1}{2}\sigma_1^2x^2\frac{\partial^2 u}{\partial^2 x^2}-\rho \sigma_1\sigma_2xy\frac{\partial^2 u}{\partial xy}-\frac{1}{2}\sigma_2^2y^2\frac{\partial^2 u}{\partial^2 y^2}+ru=0,\\
&u(0,x,y)=(w_1x+w_2y-K)^+.
\end{aligned}
\right.
\end{align*}
The above convection-diffusion equation can be solved numerically on a bounded region $(0,x_{\max})\times (0,y_{\max})$ by the standard finite difference method with the artificial boundary conditions
\begin{align*}
\left\{\begin{aligned}
&u(\tau,x,0)=c(w_1x,K,r,\tau,\sigma_1),\\
&u(\tau,0,y)=c(w_2y,K,r,\tau,\sigma_2),\\
&\frac{\partial }{\partial x}u(\tau,x_{\max},y)=w_1,\\
&\frac{\partial }{\partial y}u(\tau,x,y_{\max})=w_2,\\
\end{aligned}
\right.
\end{align*}
where
\begin{align*}
c(x,K,r,\tau,\sigma)=x\Phi(d_1)-Ke^{-r\tau}\Phi(d_2),\quad
d_{1,2} = \frac{\log \frac{x}{K}+(r\pm \frac{1}{2}\sigma^2)\tau}{\sigma \sqrt{\tau}},
\end{align*}
and $\Phi(\cdot)$ is the cumulative distribution function of the standard normal distribution.
\end{document}